\documentclass[11pt, a4paper,leqno]{amsart}
\usepackage{amsmath,amsthm,amscd,amssymb,amsfonts, amsbsy}
\usepackage{latexsym}
\usepackage{exscale}
\usepackage{mathrsfs}

\usepackage{pgf}

\usepackage[utf8]{inputenc}
\usepackage[OT2,T1]{fontenc}
\usepackage{lmodern}

\usepackage[colorlinks=true, pdfstartview=FitV, linkcolor=blue, citecolor=red, urlcolor=blue]{hyperref} 

\theoremstyle{plain}
\newtheorem{theorem}[equation]{Theorem}
\newtheorem{lemma}[equation]{Lemma}

\theoremstyle{remark}
\newtheorem{remark}[equation]{Remark}
\newtheorem{example}[equation]{Example}

\theoremstyle{definition}
\newtheorem{definition}[equation]{Definition}

\numberwithin{equation}{section}

\newcommand{\N}{\mathbb{N}}
\newcommand{\Z}{\mathbb{Z}}
\newcommand{\R}{\mathbb{R}}

\newcommand{\dd}{\mathscr{D}}
\newcommand{\car}{\mathscr{C}(\mu)}

\DeclareSymbolFont{cyrletters}{OT2}{wncyr}{m}{n}
\DeclareMathSymbol{\Sha}{\mathalpha}{cyrletters}{"58}

\renewcommand{\emptyset}{\mbox{\textup{\O}}}

\DeclareMathOperator{\supp}{supp}
\DeclareMathOperator{\sgn}{sgn}

\newcommand{\BMO}{{\rm BMO}}

\begin{document}

\allowdisplaybreaks

\title[Dyadic harmonic analysis beyond doubling measures]{Dyadic harmonic analysis
\\
beyond doubling measures}

\author[L.D. L\'opez-S\'anchez]{Luis Daniel L\'opez-S\'anchez}
\address{Luis Daniel L\'opez-S\'anchez\\
Instituto de Ciencias Matem\'aticas CSIC-UAM-UC3M-UCM\\
Consejo Superior de Investigaciones Cient{\'\i}ficas\\
C/ Nicol\'as Ca\-brera, 13-15\\
E-28049 Madrid, Spain}

\email{luisd.lopez@icmat.es}

\author[J.M. Martell]{Jos\'e Mar{\'\i}a Martell}

\address{Jos\'e Mar{\'\i}a Martell\\
Instituto de Ciencias Matem\'aticas CSIC-UAM-UC3M-UCM\\
Consejo Superior de Investigaciones Cient{\'\i}ficas\\
C/ Nicol\'as Cabrera, 13-15\\
E-28049 Madrid, Spain}

\email{chema.martell@icmat.es}

\author[J. Parcet]{Javier Parcet}

\address{Javier Parcet\\
Instituto de Ciencias Matem\'aticas CSIC-UAM-UC3M-UCM\\
Consejo Superior de Investigaciones Cient{\'\i}ficas\\
C/ Nicol\'as Cabrera, 13-15\\
E-28049 Madrid, Spain}
\email{javier.parcet@icmat.es}

\

\thanks{The authors are grateful to Jos\'e M. Conde, David Cruz-Uribe, Cristina Pereyra and Carlos P\'erez for discussions related to this paper.
Supported in part by ERC Grant StG-256997-CZOSQP, by MINECO Spanish Grant MTM-2010-16518 and by ICMAT Severo Ochoa project SEV-2011-0087.}

\date{\today}

\date{November 27, 2012. \textit{Revised}: \today}

\subjclass[2010]{42B20, 42B25, 42C40, 42C10}

\keywords{Dyadic cubes, dyadic Hilbert transform, dyadic paraproducts, generalized Haar systems, Haar shift operators, non-doubling measures, Calder\'on-Zygmund decomposition.}

\begin{abstract}
We characterize the Borel measures $\mu$ on $\R$ for which the associated dyadic Hilbert transform, or its adjoint, is of weak-type $(1,1)$ and/or strong-type $(p,p)$ with respect to $\mu$. Surprisingly, the class of such measures is strictly bigger than the traditional class of dyadically doubling measures and strictly smaller than the whole Borel class. In higher dimensions, we provide a complete characterization of the weak-type $(1,1)$ for arbitrary Haar shift operators, cancellative or not, written in terms of two generalized Haar systems and these include the dyadic paraproducts. Our main tool is a new Calder\'on-Zygmund decomposition valid for arbitrary Borel measures which is of independent interest.
\end{abstract}

\maketitle


\section{Introduction}

Dyadic techniques are nowadays fundamental in harmonic analysis. Their origin goes back to Hardy, Littlewood, Paley and Walsh among others. In the context of martingale inequalities, the dyadic maximal and square functions arise as particular cases of Doob's maximal function and Burkholder's square function for martingales associated to a dyadic filtration. Similarly, singular integral operators have been traditionally modeled by martingale transforms or martingale paraproducts. These last operators can be written in terms of martingale differences and conditional expectations, so that the full strength of probability methods applies in the analysis of their boundedness properties. In the Euclidean setting, dyadic martingale differences decompose as a sum of Haar functions and therefore we can obtain expansions using the classical Haar system.

In the last years dyadic operators have attracted a lot of attention related to the so-called $A_2$-conjecture. This seeks to establish that some operators obey an $L^2(w)$ estimate for every $w\in A_2$ with a constant that grows linearly in the $A_2$-characteristic of $w$. For the maximal function this was proved by Buckley \cite{buckley1993}. In \cite{wittwer2000}, Wittwer proved the $A_2$-conjecture for Haar multipliers in one dimension. The Beurling-Ahlfors transform, the Hilbert transform and the Riesz transforms were then considered by Petermichl and
Volberg in \cite{petermichl-volberg2002}, \cite{petermichl2007}, \cite{petermichl2008} (see also \cite{dragicevic-volberg2003}) and the $A_2$-conjecture for them was shown via the representation of these operators as averages of
Haar multipliers and certain dyadic operators called Haar shifts.  Paraproducts were treated in \cite{beznosova}, and with a different approach in \cite{cruzuribe-martell-perez2012}. The final solution to the $A_2$-conjecture for general Calder\'on-Zygmund operators was obtained by Hyt\"onen in his celebrated paper \cite{hytonen2012}. Again, a key ingredient in the proof is that Calder\'on-Zygmund operators can be written as averages of dyadic operators including Haar shift operators, dyadic paraproducts and their adjoints.

The \emph{dyadic Hilbert transform} is given by
\begin{equation*}
H_\dd f(x) = \sum_{I \in \mathscr{D}} \langle f,h_I \rangle\bigl(h_{I_-}(x) - h_{I_+}(x)\bigr).
\end{equation*}
Here $\dd$ denotes some dyadic grid in $\R$ and $h_I$ is the Haar function associated with $I\in\dd$: $h_I=|I|^{-1/2}(1_{I_-}-1_{I_+})$ where $I_-$ and $I_+$ are the left and right dyadic children of $I$. The importance of this operator comes from the fact that the classical Hilbert transform can be obtained via averaging $H_\dd$ over dyadic grids, this was shown by Petermichl \cite{petermichl2000}. That $H_\dd$ is bounded on $L^2(\R)$ follows easily from the orthogonality of the Haar system. Using the standard Calder\'on-Zygmund decomposition one can easily obtain (see for instance \cite{cruzuribe-martell-perez2012}) that $H_\dd$ is of weak-type $(1,1)$ and therefore bounded on $L^p(\R)$ for $1<p<2$. The bounds for $p>2$ can be derived by duality and interpolation from the weak-type $(1,1)$ of the adjoint operator.

Let us consider a Borel measure $\mu$ in $\R$. One can define a Haar system in a similar manner which is now orthonormal in $L^2(\mu)$. Hence, we may consider a dyadic Hilbert transform which we momentarily denote by $H_\dd^\mu$ and ask about its boundedness properties. The boundedness on $L^2(\mu)$ is again automatic by orthogonality.  The standard Calder\'on-Zygmund theory can be easily extended to settings where the underlying measure is doubling. In the present situation, since the operator is dyadic, one could even relax that condition and assume that $\mu$ is dyadically doubling. In such a case, we can almost copy verbatim the standard proof and conclude the weak-type $(1,1)$ (with respect to $\mu$) and therefore obtain the same bounds as before. Suppose next that the measure $\mu$ is not dyadically doubling, and we would like to find the class of measures $\mu$ for which $H_\dd^\mu$ maps continuously $L^1(\mu)$ into $L^{1,\infty}(\mu)$.  Characterizing the class of measures for which a given operator is bounded is in general a hard problem. For instance,  that is the case for the $L^2$ boundedness of the Cauchy integral operator in the plane and the class of linear growth measures obtained by Tolsa \cite{tolsa1999}. This led to non-standard Calder\'on-Zygmund theories (where $\mu$ has some polynomial growth \textit{\'a la} Nazarov-Treil-Volberg and Tolsa) that one could try to apply in the present situation. This would probably require some extra (and \textit{a posteriori} unnecessary) assumptions on $\mu$. On the other hand, let us recall that $H_\dd^\mu$ is a dyadic operator. Sometimes dyadic operators behave well even without assuming doubling: the dyadic Hardy-Littlewood maximal function and the dyadic square function are of weak-type $(1,1)$ for general Borel measures $\mu$,  see respectively \cite{doob} and \cite{burkholder}. In view of that, one could be tempted to conjecture that $H_\dd^\mu$ is of weak-type $(1,1)$ for general measures $\mu$ without assuming any further doubling property (or polynomial growth). One could also ask the same questions for some other dyadic operators: the adjoint of the dyadic Hilbert transform, (cancellative) Haar shift operators, dyadic paraproducts or their adjoints or, more in general, non-cancellative Haar shift operators (we give the precise definitions of these objects below). This motivates one of the main questions we address in this paper:
\begin{quote}
\em Determine the family of measures $\mu$ for which a given dyadic operator (e.g., the dyadic Hilbert transform or its adjoint, a dyadic paraproduct or its adjoint, a cancellative or non-cancellative Haar shift operator)  maps continuously $L^1(\mu)$ into $L^{1,\infty}(\mu)$.
\end{quote}
We know already that if $\mu$ is dyadically doubling these operators satisfy weak-type estimates by a straightforward use of the standard Calder\'on-Zygmund theory. Therefore, it is natural to wonder whether the doubling condition is necessary or it is just convenient. As we will see along this paper there is no universal answer to that question for all the previous operators: the class of measures depends heavily on the operator in question. Let us illustrate this phenomenon with some examples:

\begin{list}{$\bullet$}{\leftmargin=0.4cm\labelwidth=.4cm  \itemsep=0.2cm}
\item \textit{Dyadic paraproducts and $1$-dimensional  Haar multipliers.} We shall see in Theorems \ref{thm:Hil}, \ref{thm:HS-general} and \ref{theor:parap} that these operators are of weak-type $(1,1)$ for every locally finite Borel measure.

\item \textit{The dyadic Hilbert transform and its adjoint.} We shall prove in Theorem \ref{thm:Hil} that each operator gives rise to a family of measures governing the corresponding weak-type $(1,1)$. In Section \ref{section:examples} we shall provide some examples of measures, showing that the two classes (the one for the dyadic Hilbert transform and the one for its adjoint) are different and none of them is contained in the other. Further, the class of dyadically doubling measures is strictly contained in the intersection of the two classes.

\item \textit{Adjoints of dyadic paraproducts.} We shall obtain in Theorem \ref{theor:parap} that the weak-type $(1,1)$ of these operators leads naturally to the  dyadically doubling condition for $\mu$.

\end{list}


Besides these examples, our main results will answer the question above providing a characterization of the measures for which any of the previous operators is of weak-type $(1,1)$. It should be pointed out that the proof of such results are relatively simple, once we have obtained the appropriate Calder\'on-Zygmund decomposition valid for general measures. We propose a new Calder\'on-Zygmund decomposition, interesting on its own right, with a new good part which will be still higher integrable. We need to do this, since the usual ``good part'' in the classical Calder\'on-Zygmund decomposition is no longer good in a general situation: the $L^\infty$ bound (or even any higher integrability) is ruined by the fact that the average of $f$ on a given maximal cube cannot be bounded unless the measure is assumed to be doubling or dyadically doubling. This new good part leads to an additional bad term that needs to be controlled. More precisely, fixed $\lambda>0$, let $\{Q_j\}_j$ be the corresponding family of maximal  dyadic cubes (maximal with respect to the property that $\langle |f|\rangle_Q>\lambda$, see below for notation). Then we write $f=g+b+\beta$ where
\begin{list}{$\bullet$}{\leftmargin=0.4cm\labelwidth=.4cm  \itemsep=0.2cm}
\item $g\in L^p(\mu)$ for every  $1\le p<\infty $ with
$$
\|g\|_{L^p(\mu)}\le C_p\,\lambda^{p-1}\,\|f\|_{L^1(\mu)};
$$

\item $b=\sum_j b_j$, with
$$\supp(b_j) \subset Q_j, \quad \int_{\R^d}b_j(x) \,d\mu(x) = 0, \quad \sum_j \|b_j\|_{L^1(\mu)} \leq 2\,\|f\|_{L^1(\mu)};$$

\item  $\beta =  \sum_j \beta_j$, with
$$
\supp(\beta_j) \subset \widehat{Q}_j, \quad \int_{\R^d} \beta_j(x) \,d\mu(x) =0, \quad \sum_j \|\beta_j\|_{L^1(\mu)} \leq 4\,\|f\|_{L^1(\mu)},
$$
where, for each $j$, we write $\widehat{Q}_j$ to denote the dyadic parent of $Q_j$.
\end{list}
Let us compare this with the classical Calder\'on-Zygmund decomposition. First, we lose the $L^\infty$ bound for the good part, however, for practical purposes this is not a problem since in most of the cases one typically uses the $L^2$ estimate for $g$. We now have two bad terms: the typical one $b$; and the new one $\beta$, whose building blocks are supported in the dyadic cubes $\{\widehat{Q}_j\}_j$, which are not pairwise disjoint, but still possess some cancelation. This new Calder\'on-Zygmund decomposition is key to obtaining the weak-type estimates for the Haar shift operators we consider.

The organization of the paper is as follows. In Section \ref{section:main} we will state our main results and give some applications. Section \ref{section:proofs} contains the proof of our main results. In Section \ref{section:examples} we shall present some examples of measures in $\R$ that are not dyadically doubling (neither have polynomial growth) for which either the dyadic Hilbert transform, its adjoint or both are of weak-type $(1,1)$. In the higher dimensional case we will review some constructions of Haar systems. We shall see that the obtained characterization depends also on the Haar system that we work with. That is, if we take a Haar shift operator (i.e., we fix the family of coefficients) and write it with different
Haar systems, the conditions on the measure for the weak-type $(1,1)$ depend on the chosen Haar system. Finally, in Section \ref{section:further}  we present some further results including non-cancellative Haar shift operators and therefore dyadic paraproducts, and some comments about the relationship between Haar shifts and martingale transforms.

\section{Main results}\label{section:main}

In this paper we study the boundedness behavior of dyadic operators with respect to Borel measures that are not necessarily doubling.  For simplicity we will restrict ourselves to the Euclidean setting with the standard dyadic grid  $\dd $ in $\R^d$. Of course, our results should also hold for other dyadic lattices and, more in general, in the context of geometrically doubling metric spaces in terms of Christ's dyadic cubes \cite{christ1990}, or some other dyadic constructions \cite{david1991}, \cite{hytonen-martikainen2012}. We will use the following notation, for every $Q\in\dd$, we let $\dd_k(Q)$, $k\ge 1$, be the family of dyadic subcubes of side-length $2^{-k}\,\ell(Q)$.
We shall work with Borel measures $\mu$ such that $\mu(Q)<\infty$ for every dyadic cube $Q$ (equivalently, the $\mu$-measure of every compact set is finite). To go beyond the well-known framework of the Calder\'on-Zygmund theory for doubling measures, the first thing we do is to develop a Calder\'on-Zygmund decomposition adapted to $\mu$ and to the associated dyadic maximal function
$$
M_\dd  f(x) =
\sup_{x \in Q \in \dd }  \langle |f| \rangle_Q
=
\sup_{x \in Q \in \dd }
\frac{1}{\mu(Q)} \int_Q |f(x)| \, d\mu(x).
$$
Here we have used the notation $\langle g \rangle_Q$ for the $\mu$-average of $g$ on $Q$ and we set $\langle g \rangle_Q =0$ if $\mu(Q)=0$.
As usual, if $f\in L^1(\mu)$ and $\lambda>0$, we cover $\{ M_\dd  f > \lambda \}$  by the maximal dyadic cubes $\{Q_j\}_j$. In the general setting that we are considering, such maximal cubes exist (for every $\lambda>0$) if the $\mu$-measure of  every $d$-dimensional quadrant is infinity. Otherwise, maximal cubes exist for $\lambda$ large enough. For the sake of clarity in exposition, in the following result we assume that each $d$-dimensional quadrant has infinite $\mu$-measure. The general case will be addressed in Section \ref{section:general-mu} below.

One could try to use the standard Calder\'on-Zygmund decomposition, $f=g+b$ where $g$ and $b$ are respectively the ``good'' and  ``bad'' parts. As usual, in each $Q_j$ the ``good'' part would agree with $\langle f \rangle_{Q_j}$. However, this good part would not be bounded (or even higher integrable) and therefore this decomposition would be of no use. Our new Calder\'on-Zygmund decomposition solves the problem with the ``good'' part  and adds a new ``bad'' part whose building blocks have vanishing integrals and each of them is supported in $\widehat{Q}_j$, the dyadic parent of $Q_j$.

\begin{theorem} \label{theor:CZ-d}
Let $\mu$ be a Borel measure on $\R^d$ satisfying that $\mu(Q)<\infty$ for all $Q\in\dd $ and that each $d$-dimensional quadrant has infinite $\mu$-measure. Given an
integrable function $f \in L^1(\mu)$ and $\lambda>0$, consider the standard covering of $\Omega_\lambda = \{M_\dd  f > \lambda\}$ by maximal dyadic cubes $\{Q_j\}_j$. Then we can write $f = g + b + \beta$ with
\begin{align*}
g(x) &=
f(x)\,1_{\mathbb{R}^d\setminus \Omega_\lambda}(x) +  \sum_j \langle f \rangle_{\widehat{Q}_j}\,1_{Q_j}(x)
\\
&\hskip4cm
+\sum_j \big(\langle f \rangle_{Q_j} - \langle f \rangle_{\widehat{Q}_j}\big) \,\frac{\mu(Q_j)}{\mu(\widehat{Q}_j)}\, 1_{\widehat{Q}_j}(x),
\\[5pt]
b(x) & =  \sum_j b_j(x) \ = \ \sum_j \bigl(f(x) - \langle f \rangle_{Q_j}\bigr) \,1_{Q_j}(x),
\\[5pt]
\beta(x) & =  \sum_j \beta_j(x) \ = \ \sum_j \bigr(\langle f \rangle_{Q_j} - \langle f \rangle_{\widehat{Q}_j}\bigl) \, \Big(1_{Q_j}(x) - \frac{\mu(Q_j)}{\mu(\widehat{Q}_j)} \,1_{\widehat{Q}_j}(x) \Big).
\end{align*}
Then, we have the following properties:
\begin{list}{$(\theenumi)$}{\usecounter{enumi}\leftmargin=.8cm
\labelwidth=.8cm\itemsep=0.2cm\topsep=.1cm
\renewcommand{\theenumi}{\alph{enumi}}}

\item \label{item:good} The function $g$ satisfies $$\|g\|_{L^p(\mu)}^p \, \le \, C_{p} \,\lambda^{p-1} \,\|f\|_{L^1(\mu)} \quad \mbox{for every} \quad 1 \le p < \infty.$$

\item \label{item:bada} The function $b$ decomposes as $b=\sum_j b_j$, where $$\supp(b_j) \subset Q_j, \quad \int_{\R^d}b_j(x) \,d\mu(x) = 0, \quad \sum_j \|b_j\|_{L^1(\mu)} \leq 2\,\|f\|_{L^1(\mu)}.$$

\item  \label{item:badb} The function $\beta$ decomposes as $\beta =  \sum_j \beta_j$, where
$$
\supp(\beta_j) \subset \widehat{Q}_j, \quad \int_{\R^d} \beta_j(x) \,d\mu(x) =0, \quad \sum_j \|\beta_j\|_{L^1(\mu)} \leq 4\,\|f\|_{L^1(\mu)}.
$$
\end{list}
\end{theorem}

Theorem \ref{theor:CZ-d} is closely related to Gundy's martingale decomposition \cite{gundy1968} and was obtained in the unpublished manuscript \cite{martell-parcetP} (see also \cite{cruzuribe-martell-perez2011}). It is however more flexible because the building blocks are the maximal cubes in place of the martingale differences. This feature is crucial when considering Haar shift operators allowing us to characterize their weak-type $(1,1)$ for general Borel measures.

A baby model of the mentioned characterization ---which will be illustrative for the general statement--- is given by the dyadic Hilbert transform in $\R$ and its adjoint. To define this operator we first need to introduce some notation. First, to simplify the exposition, let us assume that $\mu(I)>0$ for every $I\in \dd$, below we will consider the general case.
Given $I\in\dd$ we write $I_-, I_+$ for the (left and right) dyadic children of $I$, and, as before, $\widehat{I}$ is the dyadic parent of $I$. We set
\begin{equation}\label{def-hI}
h_I \, = \, \sqrt{m(I)} \, \bigg( \frac{1_{I_-}}{\mu(I_-)} - \frac{1_{I_+}}{\mu(I_+)} \bigg), \qquad \mbox{with} \quad m(I) =  \frac{\mu(I_-)\mu(I_+)}{\mu(I)}.
\end{equation}
Let us first observe that the system $\mathcal{H}=\{h_I\}_{I\in\dd}$ is orthonormal. Additionally, for every $I\in\dd$ we have
\begin{equation}\label{hI:1-infty}
\|h_I\|_{L^1(\mu)}=2\sqrt{m(I)},
\qquad\qquad
\|h_I\|_{L^\infty(\mu)}\approx\frac1{\sqrt{m(I)}}.
\end{equation}
Therefore we obtain
the following condition which will become meaningful later
\begin{equation}\label{hI:1-infty-mult}
\sup_{I \in\dd} \|h_I\|_{L^\infty(\mu)}\|h_I\|_{L^1(\mu)}<\infty.
\end{equation}
We define the dyadic Hilbert transform by
$$
H_\dd  f(x) =
\sum_{I \in \dd } \langle f, h_I \rangle \big(h_{I_-}(x) - h_{I_+}(x)\big)
=
\sum_{I \in \dd } \sigma(I)\langle f, h_{\widehat{I}} \rangle h_I(x),
$$
where $\sigma(I)=1$ if $I=(\widehat{I}\,)_-$ and $\sigma(I)=-1$ if $I=(\widehat{I}\,)_+$.
Another toy model in the $1$-dimensional setting is the adjoint of $H_\dd$ which can be written as
$$
H_\dd ^* f(x) \, = \, \sum_{I \in \dd } \sigma(I)\langle f, h_I \rangle h_{\widehat{I}}(x).
$$

We are going to show that the increasing or decreasing properties of $m$ characterize the boundedness of $H_\dd$ and $H_\dd^*$. This motivates the following definition. We say that $\mu$ is \textit{$m$-increasing} if there exists $0<C<\infty$ such that
$$
m(I) \le C\, m(\widehat{I}\,),
\qquad
I\in\dd.
$$
We say that $\mu$ is \textit{$m$-decreasing} if there exists $0<C<\infty$ such that
$$
m(\widehat{I}\,)\le C\, m(I),
\qquad
I\in\dd.
$$
Finally, we say that $\mu$ is \textit{$m$-equilibrated} if $\mu$ is both $m$-increasing and $m$-decreasing.

Let us note that if $\mu$ is the Lebesgue measure, or in general any dyadically doubling measure, we have that $m(I)\approx \mu(I)$ and therefore $\mu$ is $m$-equilibrated. As we will show below, the converse is not true. In general, we observe that $m(I)$ is half the harmonic mean of the measures of the children of $I$ and therefore,
\begin{multline*}
m(I)
=
\left(\frac1{\mu(I_-)}+\frac1{\mu(I_+)}\right)^{-1}
\approx
\left(\max\left\{\frac1{\mu(I_-)},\frac1{\mu(I_+)}\right\}\right)^{-1}
\\[7pt]
=
\min \big\{ \mu(I_-),\mu(I_+) \big\}<\mu(I).
\end{multline*}
Thus, $m$ gives quantitative information about the degeneracy of $\mu$ over $I$: $m(I)/\mu(I)\ll 1$ implies that $\mu$ mostly concentrates on only one child of $I$, and $m(I)/\mu(I)\gtrsim 1$ gives that $\mu(I_-)\approx\mu(I_+)\approx \mu(I)$.

We are ready to state our next result which characterizes the measures for which $H_\dd$ and $H_\dd^*$ are bounded for $p\neq 2$.

\begin{theorem}\label{thm:Hil}
Let $\mu$ be a Borel measure on $\R$ satisfying that $0<\mu(I)<\infty$ for every $I\in\dd $.
\begin{list}{$(\theenumi)$}{\usecounter{enumi}\leftmargin=1cm
\labelwidth=1cm\itemsep=0.2cm\topsep=.1cm
\renewcommand{\theenumi}{\roman{enumi}}}

\item $H_\dd : L^1(\mu) \to L^{1,\infty}(\mu)$ if and only if $\mu$ is $m$-increasing.

\item $H_\dd ^*: L^1(\mu) \to L^{1,\infty}(\mu)$ if and only if $\mu$ is $m$-decreasing.
\end{list}
Moreover, if $1 < p < 2$  we have:
\begin{list}{$(\theenumi)$}{\usecounter{enumi}\leftmargin=1cm
\labelwidth=1cm\itemsep=0.2cm\topsep=.1cm
\renewcommand{\theenumi}{\roman{enumi}}} \setcounter{enumi}{2}

\item $H_\dd : L^p(\mu) \to L^p(\mu)$ if and only if $\mu$ is  $m$-increasing.

\item $H_\dd ^*: L^p(\mu) \to L^p(\mu)$ if and only if $\mu$ is  $m$-decreasing.
\end{list}
If $2 < p < \infty$, by duality, the previous equivalences remain true upon switching the conditions on $\mu$.

 Furthermore, given two non-negative integers $r$, $s$, let $\Sha_{r,s}$ be a Haar shift of complexity $(r,s)$, that is,
\begin{equation}\label{HS-d=1}
\Sha_{r,s} f (x) = \sum_{I \in \dd } \sum_{\begin{subarray}{c} J \in \dd _r(I) \\ K \in \dd _s(I) \end{subarray}} \alpha_{J,K}^I \langle f, h_J \rangle h_K(x) \quad \mbox{with} \  \ \sup_{I,J,K}|\alpha_{J,K}^I| <\infty.
\end{equation}
If $\mu$ is $m$-equilibrated then  $\Sha_{r,s}$ is bounded from $L^1(\mu)$ to $L^{1,\infty}(\mu)$ and from $L^p(\mu)$ to $L^p(\mu)$ for every $1 < p < \infty$.
\end{theorem}

\medskip

Let us observe that our assumption on the coefficients of the Haar shift operator is not standard, below we shall explain why this is natural (see Theorem \ref{thm:HS-general} and the comment following it).

Let us observe that using the notation in the previous result $H_\dd$ is a Haar shift of complexity $(0,1)$ whereas $H_\dd^*$ is a Haar shift of complexity $(1,0)$. As noted above, dyadically doubling measures are $m$-equilibrated. Therefore, in this case, $H_\dd$, $H_\dd^*$, and all $1$-dimensional Haar shifts $\Sha_{r,s}$ with arbitrary complexity are of weak-type $(1,1)$ and bounded on $L^p(\mu)$ for every $1<p<\infty$.  In Section \ref{section:examples:1D} we shall present examples of measures in $\R$ as follows:
\begin{list}{$\bullet$}{\leftmargin=0.4cm\labelwidth=.4cm  \itemsep=0.2cm}

\item $\mu$ is $m$-equilibrated, but $\mu$  is neither dyadically doubling nor of polynomial growth. Thus, we have an example of a measure that is out of the classical theory for which the dyadic Hilbert transform, its adjoint and any Haar shift is of weak-type $(1,1)$ and bounded on $L^p(\mu)$ for every $1<p<\infty$.

\item $\mu$ is $m$-increasing, but $\mu$ is not $m$-decreasing, not dyadically doubling, not of polynomial growth. Thus, $H_\dd$ is of weak-type $(1,1)$, bounded on $L^p(\mu)$ for every $1<p\le 2$ and unbounded on $L^p(\mu)$ for $2<p<\infty$; $H_\dd^*$ is bounded on $L^p(\mu)$ for $2\le p<\infty$,
    not of weak-type $(1,1)$ and unbounded on $L^p(\mu)$ for every $1<p<2$.

\item $\mu$ is $m$-decreasing, but $\mu$ is not $m$-increasing, not dyadically doubling, not of polynomial growth. Thus, $H_\dd$ is bounded on $L^p(\mu)$ for $2\le p<\infty$,
    not of weak-type $(1,1)$ and unbounded on $L^p(\mu)$ for every $1<p<2$;  $H_\dd^*$ is of weak-type $(1,1)$, bounded on $L^p(\mu)$ for every $1<p\le 2$ and unbounded on $L^p(\mu)$ for $2<p<\infty$.

\item $\mu$ is not $m$-decreasing, not $m$-increasing, not dyadically doubling, but $\mu$ has polynomial growth. Thus, this is an example of a measure \textit{\'a la} Nazarov-Treil-Volberg and Tolsa for which $H_\dd$ and $H_\dd^*$ are bounded on $L^2(\mu)$, unbounded on $L^p(\mu)$ for $1< p<\infty$, $p\neq 2$, and not of weak-type $(1,1)$.

\end{list}

Our next goal is to extend the previous result to higher dimensions. In this case we do not necessarily assume that the measures have full support. The building blocks, that is, the Haar functions are not in one-to-one correspondence to the dyadic cubes: associated to every cube $Q$ we expect to have at most $2^{d}-1$ linearly independent Haar functions. Moreover, there are different ways to construct a Haar system (see Section \ref{section-Haar-systems} below). We next define the Haar systems that we are going to use:

\begin{definition}\label{def:GHS}
Let $\mu$ be a Borel measure on $\R^d$, $d\ge 1$, satisfying that $\mu(Q)<\infty$ for every $Q\in\dd $. We say that $\Phi=\{\phi_Q\}_{Q\in\dd}$ is a \textit{generalized Haar system} in $\R^d$ if the following conditions hold:
\begin{list}{$(\theenumi)$}{\usecounter{enumi}\leftmargin=.8cm
\labelwidth=.8cm\itemsep=0.2cm\topsep=.1cm
\renewcommand{\theenumi}{\alph{enumi}}}

\item For every $Q\in\dd$, $\supp(\phi_Q)\subset Q$.

\item If $Q'$, $Q\in\dd$ and $Q'\subsetneq Q$, then $\phi_Q$ is constant on $Q'$.

\item For every $Q\in\dd$, $\displaystyle\int_{\R^d} \phi_Q(x)\,d\mu(x)=0.$

\item For every $Q\in\dd$, either $\|\phi_Q\|_{L^2(\mu)}=1$ or $\phi_Q\equiv 0$.

\end{list}

\end{definition}

\medskip

\begin{remark}\label{remarks:def-GHS}
The following comments pertain to the previous definition.

\begin{list}{$\bullet$}{\leftmargin=0.4cm\labelwidth=.4cm  \itemsep=0.2cm}

\item Note that $(b)$ implies that $\phi_Q$ is constant on the dyadic children of $Q$. In particular, $\phi_Q$ is a simple function which takes at most $2^d$ different values. 

\item Given a generalized Haar system $\Phi=\{\phi_Q\}_{Q\in\dd}$,  we write $\dd_\Phi$ for the set of dyadic cubes $Q$ for which $\phi_Q\not\equiv 0$. By assumption, we allow $\dd_\Phi$ to be a proper subcollection of $\dd$. Note that $\{\phi_Q\}_{Q\in \dd}$ is an orthogonal system whereas $\{\phi_Q\}_{Q\in \dd_\Phi}$ is orthonormal.

    Let us point out that we allow the measure $\mu$ to vanish in some dyadic cubes. If $\mu(Q)=0$, we must  have $\phi_Q\equiv0$ and therefore $Q\in\dd\setminus\dd_\Phi$. If $\mu(Q)=\mu(Q')$ for some child $Q'$ of $Q$ (i.e., every brother of $Q'$ has null $\mu$-measure) then  $\phi_Q\equiv 0$ and thus $Q\in\dd\setminus\dd_\Phi$. Suppose now that $Q\in\dd_\Phi$ (therefore $\mu(Q)>0$),  by convention, we set $\phi_Q\equiv 0$ in every dyadic child of $Q$ with vanishing measure.

\item Let us suppose that for every $Q\in\dd_{\Phi}$, $\phi_Q$ takes exactly 2 different non-zero values (call $\Phi$ a \textit{$2$-value} generalized Haar system). In view of the previous remark, $\phi_Q$ is ``uniquely'' determined modulo a multiplicative $\pm1$. That is,  we can find $E_Q^+$, $E_Q^-\subset Q$, such that $E_Q^+\cap E_Q^-=\emptyset$, $E_Q^{\pm}$ is comprised of dyadic children of $Q$, $\mu(E_Q^{\pm})> 0$ and
\begin{equation}\label{eqn:2-value-Haar}
\phi_Q
=
\sqrt{m_\Phi(Q)}
\, \bigg( \frac{1_{E_Q^-}}{\mu(E_Q^-)} - \frac{1_{E_Q^+}}{\mu(E_Q^+)} \bigg),
\quad\mbox{with}\ \
m_\Phi(Q)=\frac{\mu(E_Q^-)\mu(E_Q^+)}{\mu(E_Q^-\cup E_Q^+)}.
\end{equation}
Then, for every $Q\in\dd_\Phi$ we have
\begin{equation}\label{eqn:2-value-Haar:1-infty}
\|\phi_Q\|_{L^1(\mu)}=2\sqrt{m_\Phi(Q)},
\qquad\qquad
\|\phi_Q\|_{L^\infty(\mu)}\approx\frac1{\sqrt{m_\Phi(Q)}}.
\end{equation}

\item In dimension $1$,  if we assume as before that $\mu(I)>0$ for every $I\in\dd$, we then have that $\mathcal{H}$ defined above is a generalized Haar system in $\R$ with $\dd_\mathcal{H}=\dd$.  The previous remark and the fact every dyadic interval has two children say that $\mathcal{H}$ is ``unique'' in the following sense: let $\Phi$ be a generalized Haar system in $\R$, then $\phi_I=\pm\,h_I$ for every $I\in\dd_\Phi$. Note that we can now allow the measure to vanish on some dyadic intervals. In such a case we will have that $\phi_I\equiv 0$ for every $I\in\dd$ for which $\mu(I_-)\cdot\mu(I_+)=0$. Also, $\phi_I=\pm h_I$  and $m_\Phi(I)=m(I)$ for every $I\in\dd_\Phi$.

\end{list}
\end{remark}

\medskip

Our main result concerning general Haar shift operators characterizes the weak-type $(1,1)$ in terms of the measure $\mu$ and the generalized Haar systems that define the operator. In Section \ref{section:NC-HS} we shall also consider non-cancellative Haar shift operators where condition $(c)$ in Definition \ref{def:GHS} is dropped for the Haar systems $\Phi$ and $\Psi$. This will allow us to obtain similar results for dyadic paraproducts.

\begin{theorem}\label{thm:HS-general}
Let $\mu$ be a Borel measure on $\R^d$, $d\ge 1$, such that $\mu(Q)<\infty$ for every $Q\in\dd $.
Let $\Phi=\{\phi_Q\}_{Q\in\dd}$ and $\Psi=\{\psi_Q\}_{Q\in\dd}$ be two generalized Haar systems in $\R^d$. Given two non-negative integers $r$, $s$ we
set
$$
\Xi(\Phi,\Psi; r,s)=\sup_{Q\in\dd}\big\{\|\phi_R\|_{L^\infty(\mu)}\|\psi_S\|_{L^1(\mu)}: R \in \dd _r(Q), S \in \dd _s(Q)\big\}.
$$
Let $\Sha_{r,s}$ be a Haar shift of complexity $(r,s)$, that is,
$$
\Sha_{r,s} f (x) = \sum_{Q \in \dd } \sum_{\begin{subarray}{c} R \in \dd _r(Q) \\ S \in \dd _s(Q) \end{subarray}} \alpha_{R,S}^Q \langle f, \phi_R \rangle \psi_S(x) \qquad \mbox{with} \quad \sup_{Q,R,S}|\alpha_{R,S}^Q|<\infty.
$$
If $\Xi(\Phi,\Psi; r,s)<\infty$, then $\Sha_{r,s}$ maps continuously $L^1(\mu)$ into $L^{1,\infty}(\mu)$, and by interpolation $\Sha_{r,s}$ is bounded on $L^p(\mu)$, $1<p\le 2$.

Conversely, let $\Sha_{r,s}$ be a Haar shift of complexity $(r,s)$ satisfying the non-degeneracy condition $\inf_{Q,R,S} |\alpha^{Q}_{R,S}|>0$. If $\Sha_{r,s}$ maps continuously $L^1(\mu)$ into $L^{1,\infty}(\mu)$ then $\Xi(\Phi,\Psi; r,s)<\infty$.
\end{theorem}

Let us point out that in the Euclidean setting with the Lebesgue measure one typically assumes that $|\alpha_{R,S}^Q|\lesssim (|R|\,|S|)^{1/2}/|Q|$. Our condition, with a general measure, is less restrictive and more natural: having assumed the corresponding condition with respect to $\mu$,  $H_\dd$ and $H_\dd^*$ would not be $1$-dimensional Haar shift operators unless $\mu$ is dyadically doubling.

To illustrate the generality and the applicability of Theorem \ref{thm:HS-general} we consider some examples. Before doing that we need to introduce some notation. Let $\Phi$ be a generalized Haar system in $\R^d$,
we say that $\Phi$ is \textit{standard} if
\begin{equation}\label{def:standard}
\sup_{Q\in\dd}\|\phi_Q\|_{L^1(\mu)}\,\|\phi_Q\|_{L^\infty(\mu)}<\infty.
\end{equation}
Note that we can restrict the supremum to $Q\in\dd_\Phi$. Also, if $Q\in\dd_\Phi$, H\"older's inequality and $(d)$ imply that each term in the supremum is bounded from below by $1$.  Thus, $\Phi$ being standard says that the  previous quantity is bounded from below and from above uniformly for every $Q\in\dd_\Phi$. Notice that in the language of Theorem \ref{thm:HS-general}, $\Phi$ being standard is equivalent to $\Xi(\Phi,\Phi; 0,0)<\infty$.

\begin{remark}
\label{def:standard-2-value}
If $\Phi$ is a $2$-value generalized Haar system, \eqref{eqn:2-value-Haar:1-infty} implies that $\Phi$ is standard. Note that in $\R$ (since every dyadic interval has two children) every generalized Haar system, including $\mathcal{H}$ introduced above,  is of $2$-value type and therefore standard.
\end{remark}

\begin{example}[Haar multipliers]\label{ex:Haar-mult}
Let $\Phi=\{\phi_Q\}_{Q}$ be a generalized Haar system in $\R^d $. We take the Haar shift operator of complexity $(r,s)=(0,0)$, usually referred to as a Haar multiplier,
$$
\Sha_{0,0} f (x) = \sum_{Q \in \dd } \alpha_{Q} \langle f, \phi_Q \rangle \phi_Q(x), \qquad \mbox{with} \quad \sup_Q|\alpha_{Q}|<\infty.
$$
Then $\Xi(\Phi,\Phi; 0,0)<\infty$ is equivalent to the fact that $\Phi$ is standard.  Therefore Theorem \ref{thm:HS-general} says that
$\Sha_{0,0}$ is of weak-type $(1,1)$ provided $\Phi$ is standard. We also have the converse for non-degenerate Haar shifts of complexity $(0,0)$. As a consequence of these we have the following characterization: ``$\Phi$ is standard if and only if all Haar multipliers are of weak-type $(1,1)$''. As observed above this can be applied to any $2$-value generalized Haar system in $\R^d$. In particular, for an arbitrary measure in $\R$ such that $\mu(I)>0$ for every $I\in\dd$, all Haar multipliers of the form
$$
\Sha_{0,0} f (x) = \sum_{I \in \dd } \alpha_{I} \langle f, h_I \rangle h_I(x), \qquad \mbox{with} \quad \sup_I|\alpha_{I}|<\infty,
$$
are of weak-type $(1,1)$. In higher dimensions, taking an arbitrary measure such that $\mu(Q)>0$ for every $Q\in\dd$, any Haar multiplier as above defined in terms of a $2$-value generalized Haar system in $\R^d$ is of weak-type $(1,1)$. We note that we cannot remove the assumption that the system is $2$-value: in Section \ref{section-Haar-systems} we shall give an example of a generalized Haar system that is not standard and a Haar multiplier that is not of weak-type $(1,1)$.
All these comments can be generalized to measures without full support.
\end{example}

\begin{example}[The dyadic Hilbert transform I]
For simplicity, we first suppose that $\mu(I)>0$ for every $I\in\dd$. The dyadic Hilbert transform in $\R$ can be seen as the non-degenerate Haar shift $H_\dd=\Sha_{0, 1}$ with $\alpha_{I,I_{\pm}}^I=\mp 1$. Theorem \ref{thm:HS-general} says that $H_\dd$ is of weak-type $(1,1)$ if and only if $\Xi(\mathcal{H},\mathcal{H}; 0,1)<\infty$, which in view of \eqref{hI:1-infty} is equivalent to the fact that $\mu$  is $m$-increasing. For the adjoint of the dyadic Hilbert transform $H_\dd^*=\Sha_{1, 0}$ with $\alpha_{I_{\pm},I}^I=\mp 1$ and this is a  non-degenerate Haar shift. Again, Theorem \ref{thm:HS-general} characterizes the weak-type $(1,1)$ of $H_\dd^*$ in terms of $\Xi(\mathcal{H},\mathcal{H}; 1,0)<\infty$, which this time rewrites into the property that $\mu$ is $m$-decreasing.
\end{example}

\begin{example}[The dyadic Hilbert transform II]
We now consider the dyadic Hilbert transform but with respect to measures that may vanish. Let $\Phi$ be a generalized Haar system in $\R$ and let $\dd_\Phi$ be as before. By the discussion above we may suppose that $\phi_I=h_I$ for every $I\in\dd_\Phi$. Then, the corresponding dyadic Hilbert transform can be written as
$$
H_{\dd,\Phi}  f =
\sum_{I \in \dd } \langle f, \phi_I \rangle \big(\phi_{I_-} - \phi_{I_+}\big)
=
\!\!
\sum_{I \in \dd_\Phi: \widehat{I}\in\dd_\Phi}
\sigma(I)\langle f, h_{\widehat{I}} \rangle h_I,
$$
where $\sigma(I)=1$ if $I=(\widehat{I}\,)_-$ and $\sigma(I)=-1$ if $I=(\widehat{I}\,)_+$.
As before we have that $H_{\dd,\Phi}=\Sha_{0,1}$ is non-degenerate. Therefore its weak-type $(1,1)$ is characterized in terms of the finiteness of $\Xi(\Phi,\Phi; 0,1)$. Thus, we obtain that
$$
H_{\dd,\Phi}:L^1(\mu)\longrightarrow L^{1,\infty}(\mu)
\quad
\Longleftrightarrow
\quad
m(I)\le
C\,m(\widehat{I}\,),\ I,\widehat{I}\in\dd_\Phi.
$$
Note that the latter condition says that $\mu$ is $m$-increasing on the family $\dd_\Phi$ (so in particular the intervals with zero $\mu$-measure or those with one child of zero $\mu$-measure do not count).

For the adjoint of $H_{\dd,\Phi}$ we have
$$
H_{\dd,\Phi}^*  f(x)
=
\sum_{I \in \dd } \sigma(I)\langle f, \phi_I\rangle \phi_{\widehat{I}}\,
=
\!\!
\sum_{I \in \dd_\Phi: \widehat{I}\in\dd_\Phi}
\sigma(I)\langle f, h_{I} \rangle h_{\widehat{I}}
$$
and we can analogously obtain
$$
H_{\dd,\Phi}^*:L^1(\mu)\longrightarrow L^{1,\infty}(\mu)
\quad
\Longleftrightarrow
\quad
m(\widehat{I}\,)\le
C\,m(I),\ I,\widehat{I}\in\dd_\Phi.
$$
\end{example}

\begin{example}[Haar Shifts in $\R$]\label{ex:HS-1d}
We start with the case $\mu(I)>0$ for every $I\in\dd$. Let us consider $\Sha=\Sha_{r,s}$ as in \eqref{HS-d=1}, that is, a Haar shift operator of complexity $(r,s)$ defined in terms of the system  $\mathcal{H}$. By Theorem \ref{thm:HS-general} we know that $\Xi(\mathcal{H},\mathcal{H}; r,s)<\infty$ is sufficient (and necessary if we knew that $\Sha$ is non-degenerate) for the weak-type $(1,1)$. We can rewrite this condition as follows: $m(K)\lesssim m(J)$ for every $I\in\dd$, $J\in\dd_r(I)$, $K\in\dd_s(I)$. If $\mu$ is $m$-equilibrated then $m(J)\approx m(I)$ and $m(K)\approx m(I)$ for every $I\in\dd$, $J\in\dd_r(I)$, $K\in\dd_s(I)$. All these and \eqref{hI:1-infty-mult} give at once $\Xi(\mathcal{H},\mathcal{H}; r,s)<\infty$  for every $r$, $s\ge 0$. Thus, in dimension $1$, the fact $\mu$ is $m$-equilibrated implies that  every Haar shift operator is of weak-type $(1,1)$. We would like to recall that in Section \ref{section:examples} we shall construct measures that are $m$-equilibrated but are neither dyadically doubling nor of polynomial growth. Thus, Haar shift operators are a large family of (dyadic) Calder\'on-Zygmund operators obeying a weak-type $(1,1)$ bound with underlaying measures that do not satisfy those classical conditions.

For measures vanishing in some cubes, Theorem \ref{thm:HS-general} gives us a sufficient (and often necessary) condition. However, it is not clear whether in such a case one can write that condition in terms of $\mu$ being $m$-equilibrated. We would need to be able to compare $m(K)$ and $m(J)$ for $K$ and $J$ as before with the additional condition that $J$, $K\in\dd_\Phi$. Note that the fact that $\mu$ is $m$-equilibrated gives information about jumps of order $1$ in the generations and it could happen that we cannot ``connect'' $J$ and $K$ with ``1-jumps'' within $\dd_\Phi$. Take for instance $I=[0,1)$, $J=[0,4)$, $d\mu(x)=1_{[0,1)\cup [2,4)}(x)\,dx$, $\Phi=\{h_I, h_J\}$ and $\Sha_{2,0}=\langle f,h_I\rangle h_J$. Then Theorem \ref{thm:HS-general} says that $\Sha_{2,0}$ is of weak-type $(1,1)$ since
$\Xi(\Phi,\Phi; 2,0)=4\,(m[0,4)\cdot m[0,1))^{1/2}=4/\sqrt{6}<\infty$. However, $\dd_\Phi=\{I,J\}$ and these two dyadic intervals are 2-generation separated.
\end{example}

\begin{example}[Haar Shifts in $\R^d$ for $2$-value generalized Haar systems]
Let us suppose that $\Phi$ and $\Psi$ are $2$-value generalized Haar systems. Write $E_Q^{\pm}$ (resp. $F_Q^{\pm}$) for the sets associated with $\phi_Q\in\dd_\Phi$ (resp. $\psi_Q\in\dd_\Psi$), see \eqref{eqn:2-value-Haar}. By \eqref{eqn:2-value-Haar:1-infty} we have that $\Xi(\Phi,\Psi; r,s)<\infty$ if an only if $\mu$ satisfies
\begin{equation}\label{r-s:2-valued}
m_\Psi(S)
=
\frac{\mu(F_S^-)\mu(F_S^+)}{\mu(F_S^-\cup F_S^+)}
\lesssim
\frac{\mu(E_R^-)\mu(E_R^+)}{\mu(E_Q^-\cup E_R^+)}=m_\Phi(R)
\end{equation}
for every $Q\in\dd$, $R\in\dd_r(Q)$, $S\in\dd_s(Q)$, $R\in\dd_\Phi$ and $S\in\dd_\Psi$. Therefore Theorem \ref{thm:HS-general} says that $\Sha_{r,s}$ is of weak-type $(1,1)$ provided $\mu$ satisfies the condition \eqref{r-s:2-valued}. The converse holds provided $\Sha_{r,s}$ is non-degenerated.

\end{example}

\section{Proofs of the main results}\label{section:proofs}

Before proving our main results and for later use, we observe that for any measurable set $E\subset\R^d$ we have $\|1_{E}\|_{L^{1,\infty}(\mu)}=\|1_{E}\|_{L^1(\mu)}=\mu(E)$. This easily implies that if $f$ is a simple function, then
\begin{equation}\label{L1-weakL1}
\|f\|_{L^{1,\infty}(\mu)}
\le
\|f\|_{L^1(\mu)}
\le
\#\{f(x):x\in\R^d\}\,
\|f\|_{L^{1,\infty}(\mu)}.
\end{equation}

\subsection{A new Calder\'on-Zygmund decomposition}\label{section:CZ}

As pointed out before, we shall work with the standard dyadic filtration  $\dd  = \bigcup_{k \in \Z} \dd _k$ in $\R^d$, but all our results hold for any other dyadic lattice. If $k \ge 0$ is a nonnegative integer, we write $\dd_k(Q)$ for the partition of $Q$ into dyadic subcubes of side-length $2^{-k}\ell(Q)$ and $Q^{(k)}$ for its $k$-th dyadic ancestor, i.e., the only cube of side-length $2^k \ell(Q)$ that contains $Q$. The cubes in $\dd _1(Q)$ are called dyadic children of $Q$ and $\widehat{Q} = Q^{(1)}$ is the dyadic parent of $Q$.

By $\mu$ we will denote any positive Borel measure on $\mathbb{R}^d$ such that $\mu(Q) < \infty$ for all $Q \in \dd $. Write $\mathcal{M}$ for the class of such measures.  Once $\mu$ is fixed, we set for $Q \in \dd $
$$
\langle f \rangle_Q \, = \, \frac{1}{\mu(Q)} \int_Q f(x) \,d\mu(x) \qquad \mbox{with} \quad \langle f \rangle_Q = 0 \ \, \mbox{when} \ \, \mu(Q) = 0.
$$ The dyadic maximal operator for $\mu \in \mathcal{M}$ is then $M_\dd f(x)=\sup_{x\in Q\in\dd} \langle |f| \rangle_Q$.

Let us write $\mathbb{R}^d_j$, $1\le j\le 2^{d}$, for the $d$-dimensional quadrants in $\mathbb{R}^d$. It will be convenient to consider temporarily the subclass $\mathcal{M}_\infty$ of measures $\mu \in \mathcal{M}$ such that $\mu(\R_j^d) = \infty$ for all $1 \le j \le 2^d$. We will prove our main results under the assumption that $\mu \in \mathcal{M}_\infty$ and sketch in Section \ref{section:general-mu} the modifications needed to adapt our arguments for any $\mu \in \mathcal{M}$.

Assuming now that $\mu \in \mathcal{M}_\infty$, we know that $\langle |f| \rangle_Q \to 0$ as $\ell({Q}) \to \infty$ whenever $f \in L^1(\mu)$. In particular, given any $\lambda>0$, there exists a collection of disjoint maximal dyadic cubes $\{Q_j\}_j$  such that
$$
\Omega_\lambda = \big\{ x \in \R^d: \ M_{\dd } f(x) >\lambda \big\} \, = \, \bigcup_j Q_j,
$$
where the cubes $\{Q_j\}_j$ are maximal in the sense that for all dyadic cubes $Q\supsetneq Q_j$ we have
\begin{equation}\label{max-CZ}
\langle |f| \rangle_Q \leq \lambda < \langle |f| \rangle_{Q_j},
\end{equation}
Using this covering of the level set $\Omega_\lambda$, we can reproduce the classical estimate to show the weak-type $(1,1)$ boundedness of the dyadic Hardy-Littlewood maximal operator. Note that maximal cubes have positive measure by construction.

\begin{proof}[Proof of Theorem \ref{theor:CZ-d}]

We are currently assuming that $\mu \in \mathcal{M}_\infty$, see Section \ref{section:general-mu} for the modifications needed in the general case. By construction, $f = g + b + \beta$. Moreover, the support and mean-zero conditions for $b_j$ and $\beta_j$ can be easily checked. On the other hand, since the cubes $Q_j$ are pairwise disjoint
$$
\sum_j \|b_j\|_{L^1(\mu)} \le 2\,\sum_j \int_{Q_j} |f(x)| \, d\mu(x) \le 2\,\|f\|_{L^1(\mu)}.
$$
Similarly, by the maximality of the Calder\'on-Zygmund cubes, see \eqref{max-CZ}, we obtain
$$
\sum_j \|\beta_j\|_{L^1(\mu)} \le \sum_j 2 \big( \langle |f| \rangle_{Q_j} + \langle |f| \rangle_{\widehat{Q}_j} \big) \mu(Q_j) \le
4\sum_j \int_{Q_j}\!\! |f| \, d\mu \le 4\|f\|_{L^1(\mu)}.
$$
It remains to prove the norm inequalities for $g$. Write $g_1$, $g_2$ and $g_3$ for each of the terms defining $g$ and let us estimate these in turn. It is immediate that $\|g_1\|_{L^1(\mu)}\leq \|f\|_{L^1(\mu)}$. Since $M_{\dd }$ is of weak-type $(1,1)$, Lebesgue's differentiation theorem yields $\|g_1\|_{L^\infty(\mu)} \le \|M_{\dd } f \cdot 1_{\mathbb{R}^d\setminus \Omega_\lambda}\|_{L^\infty(\mu)} \le \lambda$.  The estimates for $g_2$ are similar.  Since $\langle |f| \rangle_{\widehat{Q}_j} \leq \lambda$, we obtain
$$
\|g_2\|_{L^1(\mu)} \le \lambda\,\mu(\Omega_\lambda) \le \|f\|_{L^1(\mu)} \qquad \mbox{and} \qquad \|g_2\|_{L^\infty(\mu)} \leq \lambda.
$$
These estimates immediately yield the corresponding $L^p(\mu)$-estimates for $g_1$ and $g_2$.

The estimate for $g_3$ is not straightforward: each term in the sum is supported in $\widehat{Q}_j$, and these sets are not pairwise disjoint in general. In particular, an $L^\infty$ estimate is not to be expected. However, we do have that
\begin{multline*}
|g_3(x)| \le \sum_j \big( \langle |f| \rangle_{Q_j} + \langle |f| \rangle_{\widehat{Q}_j} \big) \frac{\mu(Q_j)}{\mu(\widehat{Q}_j)} \,1_{\widehat{Q}_j}(x)
\\
\le
2\,\sum_j \Big(\int_{Q_j} |f(y)| \,d\mu(y)\Big)\,\frac{1}{\mu(\widehat{Q}_j)}  \,1_{\widehat{Q}_j}(x)
=: 2\, Tf(x).
\end{multline*}

The following lemma contains the relevant estimates for $T$:
\begin{lemma}\label{lem:aux}
Let $\{Q_j\}_j$ be a family of pairwise disjoint dyadic cubes and set
$$
Tf(x) = \sum_j \Big(\int_{Q_j} |f(y)|\,d\mu(y)\Big) \frac{1}{\mu(\widehat{Q}_j)}\,1_{\widehat{Q}_j}(x).
$$
For every $m \in \mathbb{N}$, $T$ satisfies the estimate
\begin{align*}
\|T f\|_{L^m(\mu)}^m
&\leq m! \Big(\sup_j \frac{1}{\mu(\widehat{Q}_j)}\int_{\widehat{Q}_j} |f(y)|\,d\mu(y)\Big)^{m-1}\int_{\bigcup_j Q_j} |f(x)|\,d\mu(x)
\end{align*}
\end{lemma}

Assume this result momentarily. The case $m=1$ implies $\|g_3\|_{L^1(\mu)} \le 2\,\|f\|_{L^1(\mu)}$. On the other hand, applying it for a general integer $m$, we get by \eqref{max-CZ}
$$\|g_3\|_{L^m(\mu)}^m \, \le \, 2^m\,m!\,\lambda^{m-1}\,\|f\|_{L^1(\mu)}.$$ Now, if $1 < p < \infty$ is not an integer, we take $m = [p]+1$ and let $0 < \theta < 1$ be such that $p = \theta + (1 - \theta)m$. Then, by H\"older's inequality with indices $\frac{1}{\theta}$ and $\frac{1}{1-\theta}$, we obtain as desired
$$
\|g_3\|_{L^p(\mu)}^p
\le \|g_3\|_{L^1(\mu)}^\theta \|g_3\|_{L^m(\mu)}^{(1 - \theta)m}
\le 2^p (m!)^{\frac{p - 1}{m-1}} \lambda^{p - 1}\|f\|_{L^1(\mu)}.
$$
\end{proof}

\begin{proof}[Proof of Lemma \ref{lem:aux}]
The case $m = 1$ is trivial. Let us proceed by induction and assume that the estimate for $m$ holds. Write $\varphi_j = \frac{1}{\mu(\widehat{Q}_j)} \int_{Q_j} |f|\, d\mu$ and define the sets
\begin{multline*}
\Lambda_k
=
\Big \{ (j_1, j_2, \ldots , j_{m+1}) \in \N^{m+1}: \widehat{Q}_{j_k} = \widehat{Q}_{j_1} \cap \widehat{Q}_{j_2} \cap \cdots \cap \widehat{Q}_{j_{m+1}} \Big\}
\\
=
\Big \{ (j_1, j_2, \ldots , j_{m+1}) \in \N^{m+1}:  \widehat{Q}_{j_k} \subset \widehat{Q}_{j_1}, \dots, \widehat{Q}_{j_{m+1}}\Big\}.
\end{multline*}
By symmetry we obtain
\begin{align*}
&\|T f\|_{L^{m+1}(\mu)}^{m+1}
\le
\sum_{k = 1}^{m+1} \sum_{\Lambda_{k}} \varphi_{j_1} \cdots \varphi_{j_{m + 1}} \, \mu(\widehat{Q}_{j_1} \cap \cdots \cap \widehat{Q}_{j_{m+1}})
\\
&
\quad=(m+1)
\sum_{\Lambda_{m+1}} \varphi_{j_1} \cdots \varphi_{j_m} \int_{Q_{j_{m+1}}}|f(x)|\,d\mu(x)
\\&
\quad= (m+1)\sum_{j_1, \ldots, j_m} \, \varphi_{j_1} \cdots \varphi_{j_m} \,\sum_{j_{m+1}: (j_1,\dots,j_{m+1})\in\Lambda_{m+1}} \int_{Q_{j_{m+1}}}|f(x)|\,d\mu(x).
\end{align*}
Notice that for a fixed $m$-tuple $(j_1,\dots,j_m)$, it follows that
$$
\bigcup_{j_{m+1}: (j_1,\dots,j_{m+1})\in\Lambda_{m+1}} Q_{j_{m+1}}
\subset
\bigcup_{j_{m+1}: (j_1,\dots,j_{m+1})\in\Lambda_{m+1}} \widehat{Q}_{j_{m+1}}
\subset
\widehat{Q}_{j_1} \cap \cdots \cap \widehat{Q}_{j_m},
$$
and, moreover, the cubes in the first union are pairwise disjoint.
Thus, the fact that $\widehat{Q}_{j_1} \cap \cdots \cap \widehat{Q}_{j_m}=\widehat{Q}_{j_i}$, for some $1\le i\le m$, gives
\begin{multline*}
\|T f\|_{L^{m+1}(\mu)}^{m+1}
\le (m+1)\sum_{j_1, \ldots, j_m} \, \varphi_{j_1} \cdots \varphi_{j_m} \, \int_{\widehat{Q}_{j_1} \cap \cdots \cap \widehat{Q}_{j_m}}|f(x)|\,d\mu(x)
\\
\leq
(m+1)
\Big( \sup_j \frac{1}{\mu(\widehat{Q}_j)} \int_{\widehat{Q}_j} |f| \, d\mu \Big)\sum_{j_1, \ldots, j_m } \varphi_{j_1} \cdots \varphi_{j_m} \, \mu(\widehat{Q}_{j_1} \cap \cdots \cap \widehat{Q}_{j_m})
\\
=
(m+1)
\Big( \sup_j \frac{1}{\mu(\widehat{Q}_j)} \int_{\widehat{Q}_j} |f| \, d\mu \Big) \|T f\|_{L^{m}(\mu)}^{m}.
\end{multline*}
This and the induction hypothesis yield at once the desired estimate and the proof is complete. 
\end{proof}

The new Calder\'on-Zygmund decomposition in Theorem \ref{theor:CZ-d} can be used to obtain that some classical operators are of weak-type $(1,1)$ for general Borel measures: the $\ell^q$-valued dyadic Hardy-Littlewood maximal function with $1<q<\infty$, the dyadic square function, and $1$-dimensional Haar multipliers. For the first operator,  one needs a straightforward sequence-valued extension of the new Calder\'on-Zygmund decomposition and the reader is referred to \cite{cruzuribe-martell-perez2011}. Let us then look at the dyadic square function
$$
\mathcal{S}f(x) = \bigg( \sum_{Q \in \dd} \big|\langle f \rangle_Q - \langle f \rangle_{\widehat{Q}}\big|^2 1_Q(x) \bigg)^{1/2}.
$$
It is well-known that $\mathcal{S}$ is bounded from $L^1(\mu)$ to $L^{1,\infty}(\mu)$ with a proof adopting a probabilistic point of view. However, using our Calder\'on-Zygmund decomposition one can reprove this result using harmonic analysis techniques as follows. We decompose $f=g+b+\beta$ as in Theorem \ref{theor:CZ-d}. The estimate for the good part is standard using that $\mathcal{S}$ is bounded on $L^2(\mu)$ and $(a)$ in Theorem \ref{theor:CZ-d}. For the bad terms, using the weak-type $(1,1)$ of $M_\dd$, it suffices to restrict the level set to $\R^d\setminus\Omega_\lambda$. Theorem \ref{theor:CZ-d} parts $(b)$ and $(c)$ yield respectively that $(\mathcal{S} b_j)\,1_{\R^d\setminus Q_j}\equiv 0$  and $(\mathcal{S} \beta_j)\,1_{\R^d\setminus \widehat{Q}_j}\equiv 0$. Thus everything is reduced to the following
\begin{multline*}
\mu\{x\in\R^d\setminus \Omega_\lambda: \mathcal{S}\beta(x)>\lambda/2\}
\le
\frac2{\lambda}\,\sum_j \int_{\widehat{Q}_j\setminus Q_j}|\mathcal{S}\beta_j|\,d\mu
\\
=
\frac2{\lambda}\,\sum_j |\langle f \rangle_{Q_j} - \langle f \rangle_{\widehat{Q}_j}|\frac{\mu(Q_j)}{\mu(\widehat{Q}_j)} \mu(\widehat{Q}_j \setminus Q_j)
\le
\frac4{\lambda}\,\sum_j \int_{Q_j}|f|\,d\mu
\le
\frac4{\lambda}\|f\|_{L^1(\mu)}.
\end{multline*}
All these ingredients allow one to conclude that $\mathcal{S}$ is of weak-type $(1,1)$. Details are left to the reader

Finally, under the assumption that $0<\mu(I)<\infty$ for all $I\in\dd$, we consider the $1$-dimensional Haar multipliers  defined as
$$
T_\alpha f(x)=\sum_{I \in \dd} \alpha_I \langle f,h_I \rangle  h_I(x),
\qquad
\sup_I |\alpha_I|<\infty.
$$
This operator is bounded on $L^2(\mu)$ by orthonormality. A probabilistic point of view, see Section \ref{Paragraph-Probability},  yields that $T_\alpha$ is a dyadic martingale transform and therefore of weak-type $(1,1)$. Again, our new decomposition gives a proof with a ``harmonic analysis'' flavor. We first observe that $T_\alpha b_j(x)=0$ for every $x\in \R\setminus Q_j$. Therefore, using Theorem \ref{theor:CZ-d} and proceeding as above everything reduces to the following estimate
\begin{multline*}
\mu\{x\in\R: |T_\alpha\beta(x)|>\lambda/2\}
\le
\frac2\lambda\sum_j |\alpha_{\widehat{I}_j}|\,\big|\langle f \rangle_{I_j} - \langle f \rangle_{\widehat{I}_j}\big| \sqrt{m(\widehat{I}_j)}\|h_{\widehat{I}_j}\|_{L^1(\mu)}
\\
\le
\sup_I|\alpha_I|\frac8\lambda\sum_j \langle |f| \rangle_{I_j}m(\widehat{I}_j)
\le
\sup_I|\alpha_I| \frac8{\lambda}\|f\|_{L^1(\mu)},
\end{multline*}
where we have used \eqref{beta-haar} below, \eqref{max-CZ}, \eqref{hI:1-infty} and that $m(\widehat{I}_j)<\mu(I_j)$ .

\subsection{The dyadic Hilbert transform}

In this section we prove Theorem \ref{thm:Hil}. Although the estimates for $H_\dd$ and $H_\dd^*$ follow from Theorem \ref{thm:HS-general} as explained above, we believe that it is worth giving the argument: the proofs for our toy models $H_\dd$ and $H_\dd^*$ are much simpler and have motivated our general result. We will skip, however, the last statement in the result since it follows from Theorem \ref{thm:HS-general}, as explained in Example \ref{ex:HS-1d}, and interpolation.

Before starting the proof we observe that by the orthonormality of the system $\mathcal{H}$ we have
\begin{equation}
\label{eq:dhill2}
\|H_\dd  f\|_{L^2(\mu)}^2
=
\sum_{I \in \dd } |\langle f, h_{\widehat{I}} \rangle |^2
\le 2\|f\|_{L^2(\mu)}^2.
\end{equation}
Thus, $H_\dd$ and $H_\dd^*$ are bounded on $L^2(\mu)$.

\begin{proof}[Proof of Theorem \ref{thm:Hil}, part $(i)$]
We first prove the necessity of $\mu$ being $m$-increasing. Take $f = h_I$ so that $H_\dd  f = h_{I_-} - h_{I_+}$. Using that $h_I$ is constant on dyadic subintervals of $I$, \eqref{L1-weakL1} and that $H_\dd $ is of weak-type $(1,1)$ we obtain that $\mu$ is $m$-increasing:
\begin{multline*}
\Bigl(\sqrt{m(I_-)} + \sqrt{m(I_+)}\,\Bigr)
\approx
\|h_{I_-}\|_{L^1(\mu)} + \|h_{I_+}\|_{L^1(\mu)}
\\
\approx
\|H_\dd  h_I\|_{L^{1,\infty}(\mu)}
\lesssim \|h_I\|_{L^1(\mu)}  \approx \sqrt{m(I)}.
\end{multline*}

Next we obtain that if $\mu$ is $m$-increasing then $H_\dd$ is of weak-type $(1,1)$. In order to use Theorem \ref{theor:CZ-d}, we shall assume that $\mu\in\mathcal{M}_\infty$, that is, $\mu[0,\infty)=\mu(-\infty,0)=\infty$. The general case will be considered in Section \ref{section:general-mu} below. Fix $\lambda>0$ and decompose $f$ by means of the Calder\'on-Zygmund decomposition in Theorem \ref{theor:CZ-d}. Hence,
\begin{multline*}
\mu \{x \in \mathbb{R}: |H_\dd  f(x)| > \lambda\}
\leq \mu\{ x\in \mathbb{R} : |H_\dd  g(x)| >
\lambda/3\} + \mu(\Omega_\lambda)
\\
+ \mu\{ x\in \mathbb{R} \setminus \Omega_\lambda: |H_\dd  b(x)| > \lambda/3\}
+ \mu\{ x\in \mathbb{R} : |H_\dd  \beta(x)| > \lambda/3\}
\\
= S_1 + S_2 +S_3 + S_4.
\end{multline*}
Using the weak-type $(1,1)$ for $M_\dd$, Theorem \ref{theor:CZ-d} part $(a)$ and \eqref{eq:dhill2} it is standard to check that $S_1 + S_2  \leq (C/\lambda) \|f\|_{L^1(\mu)}$.
Using that each $b_j$ has vanishing integral and that $h_I$ is constant on each $I_{\pm}$ it is easy to see that $H_\dd  b_j(x)=0$ whenever $x\in \mathbb{R}\setminus I_j$ and thus $S_3 = 0$. To estimate $S_4$ we first observe that \begin{equation}\label{beta-haar}
\langle \beta_j, h_I\rangle
=
\sigma(I_j)(\langle f \rangle_{I_j} - \langle f \rangle_{\widehat{I}_j}\bigr)\sqrt{m(\widehat{I}_j)}\,\delta_{\widehat{I}_j,I}.
\end{equation}
This can be easily obtained using that $\beta_j$ and $h_I$ have vanishing integrals; that $\beta_j$ is supported on $\widehat{I}_j$ and constant on each dyadic children of $\widehat{I}_j$; and that $h_I$ is supported on $I$. Thus,
$$
H_\dd  \beta_j
=
\sum_{I \in \dd } \sigma(I) \langle \beta_j,h_{\widehat{I}} \rangle h_I
=
(\langle f \rangle_{I_j} - \langle f \rangle_{\widehat{I}_j}\bigr)\sqrt{m(\widehat{I}_j)}\ (h_{I_j} - h_{I_j^b}),
$$
where $I^b = \widehat{I} \setminus I \in \dd $ is the dyadic brother of $I \in \dd $. Using \eqref{max-CZ}, \eqref{hI:1-infty}, the assumption that $\mu$ is $m$-increasing and the fact that $m(\widehat{I}\,) \leq \mu(I)$ for every $I \in \dd $ we conclude as desired
\begin{align*}
S_4
\leq \frac{3}{\lambda}\sum_j
\|H_\dd  \beta_j\|_{L^1(\mu)}
\lesssim
\frac{1}{\lambda}\sum_j\langle |f| \rangle_{I_j} m(\widehat{I}_j)
\lesssim \frac{1}{\lambda} \sum_j\int_{I_j} |f|\,d\mu
\leq \frac{1}{\lambda} \|f\|_{L^1(\mu)}.
\end{align*}
This completes the proof of $(i)$.
\end{proof}

\begin{proof}[Proof of Theorem \ref{thm:Hil}, part $(ii)$]
Take $f = h_I$ so that $H_\dd^*f = \sigma(I)\,h_{\widehat{I}}$.  Assuming that $H_\dd^*$ is of weak-type $(1,1)$ we obtain by \eqref{L1-weakL1} that $\mu$ is $m$-decreasing:
$$
2\sqrt{m(\widehat{I}\,)}
=
\|h_{\widehat{I}}\|_{L^1(\mu)}
\approx
\|h_{\widehat{I}}\|_{L^{1,\infty}(\mu)}
=
\|H_\dd^*f \|_{L^{1,\infty}(\mu)}
\lesssim
\|h_I\|_{L^1(\mu)} \approx\sqrt{m(I)}.
$$

To prove the converse we proceed as above. We shall assume that $\mu\in\mathcal{M}_\infty$, the general case will be considered in Section \ref{section:general-mu} below. The estimates for
$S_1$ and $S_2$ are standard (since $H_\dd^*$ is bounded on $L^2(\mu)$). For $S_3$ we first observe that if $x\in\R\setminus I_j$
$$
H_\dd^*b_j(x)
\!=\! \sum_{I \in \dd } \sigma(I) \langle b_j,h_I \rangle h_{\widehat{I}}(x)
= \sigma(I_j)\langle b_j, h_{I_j} \rangle h_{\widehat{I}_j}(x)
= \sigma(I_j)\langle f, h_{I_j} \rangle h_{\widehat{I}_j}(x).
$$
We use this expression, \eqref{hI:1-infty} and that $\mu$ is $m$-decreasing:
\begin{multline*}
S_3 \leq \frac{3}{\lambda} \sum_j \int_{\mathbb{R} \setminus I_j}|H_\dd ^* b_j(x)| \,d\mu(x)
\le
\frac{3}{\lambda} \sum_j
\|h_{I_j}\|_{L^\infty(\mu)}\|h_{\widehat{I}_j}\|_{L^1(\mu)} \int_{I_j}|f|\,d\mu
\\
\approx \frac{1}{\lambda} \sum_j \sqrt{\frac{m(\widehat{I}_j)}{m(I_j)}}\int_{I_j}|f|\,d\mu
\lesssim
\frac{1}{\lambda} \sum_j \int_{I_j}|f|\,d\mu
\le \frac1{\lambda}\|f\|_{L^1(\mu)}.
\end{multline*}
To estimate $S_4$ we use \eqref{beta-haar},
$$
H_\dd ^*\beta_j
= \sum_{I \in \dd } \sigma(I) \langle \beta_j,h_I \rangle h_{\widehat{I}}
= \sigma(\widehat{I}_j)\sigma(I_j)\bigl(\langle f \rangle_{I_j} - \langle f \rangle_{\widehat{I}_j}\bigr)\sqrt{m(\widehat{I}_j)} \, h_{I_j^{(2)}},
$$
where we recall that $I_j^{(2)}$ is the 2nd-dyadic ancestor of $I_j$.
We use that $\mu$ is $m$-decreasing and $m(\widehat{I}\,) \le \mu(I)$ to conclude that
\begin{multline*}
S_4 \leq \frac{3}{\lambda}\sum_j \|H_\dd ^* \beta_j\|_{L^1(\mu)}
\le
\frac{12}{\lambda}\sum_j \langle |f| \rangle_{I_j}\sqrt{m(\widehat{I}_j) m\bigl(I_j^{(2)}\bigr)}
\\
\lesssim
\frac{1}{\lambda}\sum_j \langle |f| \rangle_{I_j} m(\widehat{I}_j)
\lesssim \frac{1}{\lambda} \sum_j\int_{I_j} |f|\,d\mu
\leq \frac{1}{\lambda} \|f\|_{L^1(\mu)}.
\end{multline*}
This completes the proof of $(ii)$.
\end{proof}

\begin{proof}[Proof of Theorem \ref{thm:Hil}, part $(iii)$]
If  $\mu$ is $m$-increasing we can use $(i)$ to interpolate with the $L^2(\mu)$ bound to conclude estimates on $L^p(\mu)$ for every $1<p<2$. Conversely, we note that
\begin{equation}\label{hI-Lp}
\|h_I\|_{L^p(\mu)}
=
\sqrt{m(I)}\Biggl( \frac{1}{\mu(I_-)^{p-1}} + \frac{1}{\mu(I_+)^{p-1}}\Biggr)^{\frac1p}
\approx
m(I)^{\frac12 - \frac1{p'}}.
\end{equation}
On the other hand, if we then assume that $H_\dd$ is bounded on $L^p(\mu)$ we conclude that
\begin{multline*}
m(I_-)^{\frac12 - \frac1{p'}} + m(I_+)^{\frac12 - \frac1{p'}}
\approx
\|h_{I_-} - h_{I_+}\|_{L^{p}(\mu)}
=
\|H_\dd  h_I\|_{L^{p}(\mu)}
\\
\lesssim
\|h_I\|_{L^p(\mu)}
\approx m(I)^{\frac12 - \frac1{p'}}.
\end{multline*}
This and the fact that $1<p < 2$ imply that $\mu$ is $m$-increasing.
\end{proof}

\begin{proof}[Proof of Theorem \ref{thm:Hil}, part $(iv)$]
For $H_\dd ^*$ we can proceed in the same way. By interpolation and $(ii)$, $\mu$ being $m$-decreasing gives boundedness on $L^p(\mu)$ for $1<p<2$. Conversely,
if $H_\dd ^*$ is bounded on $L^p(\mu)$ for some $1 < p < 2$, then
\begin{align*}
m(\widehat{I}\,)^{\frac12 - \frac1{p'}} \approx \|h_{\widehat{I}}\|_{L^{p}(\mu)} = \|H_\dd ^* h_I\|_{L^{p}(\mu)} \lesssim \|h_I\|_{L^p(\mu)} \approx m(I)^{\frac12 - \frac1{p'}},
\end{align*}
and therefore $\mu$ is $m$-decreasing.
\end{proof}


\subsection{Haar shift operators in higher dimensions}\label{section:HS-general}

We first see that $\Sha_{r,s}$ is a bounded operator on $L^2(\mu)$.
Following  \cite{hytonen2011p} or \cite{hytonen2012}, we write
$$
\Sha_{r,s} f (x) =
\sum_{Q \in \dd } \Biggl(\sum_{\begin{subarray}{c} R \in \dd _r(Q) \\ S \in \dd _s(Q) \end{subarray}} \alpha_{R,S}^Q \langle f, \phi_R \rangle \psi_S(x)\Biggr)
=:
\sum_{Q \in \dd } A_Q f(x)
$$
As observed before, $\Phi$ and $\Psi$ are orthogonal systems. This implies
\begin{align}\label{L2-AQ}
\|A_Qf\|_{L^2(\mu)}^2
&=
\sum_{S \in \dd_s(Q)} \Bigg| \Biggl\langle f, \sum_{R \in \dd_r(Q)} \alpha_{R,S}^Q \phi_R\Biggr\rangle \Bigg|^2 \|\psi_S\|_{L^2(\mu)}^2
\\ \nonumber
&\le \|f\|_{L^2(\mu)}^2\,\sum_{\begin{subarray}{c} R \in \dd _r(Q) \\ S \in \dd _s(Q) \end{subarray}} \big|\alpha_{R,S}^Q\big|^2 \|\phi_R\|_{L^2(\mu)}^2
\\ \nonumber
&\le
2^{(r+s)d}\,\big(\sup_{Q,R, S} \big|\alpha_{R,S}^Q\big|^2\big)\|f\|_{L^2(\mu)}^2.
\end{align}

For $Q \in \dd$ and non-negative integer $r,s$, we write $P_{\Phi,Q}^r$ and $P_{\Psi,Q}^s$ for the projections
$$
P_{\Phi,Q}^r f = \sum_{R \in \dd_r(Q)} \langle f,\phi_R \rangle \phi_R,
\qquad
P_{\Psi,Q}^s f = \sum_{S \in \dd_s(Q)} \langle f,\psi_S \rangle \psi_S.
$$
We then have
\begin{multline*}
P_{Q,\Psi}^sA_QP_{Q,\Phi}^rf
=
\sum_{\begin{subarray}{c} R \in \dd _r(Q) \\ S \in \dd _s(Q) \end{subarray}} \langle A_Q(\phi_R), \psi_S \rangle \,\langle f, \phi_R \rangle \psi_S
\\
=
\sum_{\begin{subarray}{c} R \in \dd _r(Q) \\ S \in \dd _s(Q) \end{subarray}}\!\! \|\phi_R\|_{L^2(\mu)}^2\|\psi_S\|_{L^2(\mu)}^2\alpha_{R,S}^Q \langle f, \phi_R \rangle \psi_S
=
A_Q f.
\end{multline*}
Fixed $r$ and $s$, we notice that the projections $P_{\Phi,Q}^r$ are orthogonal on the index $Q$ and the same occurs with $P_{\Psi,Q}^s$. Hence, by \eqref{L2-AQ} and orthogonality
\begin{multline*}
\|\Sha_{r,s} f\|_{L^2(\mu)}^2
=
\sum_{Q \in \dd}\big\|P_{Q,\Psi}^sA_QP_{Q,\Phi}^rf\big\|_{L^2(\mu)}^2
\le
C\,
\sum_{Q \in \dd}\big\|P_{Q,\Phi}^rf\big\|_{L^2(\mu)}^2
\\
\le
C\,
\sum_{Q \in \dd}\sum_{R\in\dd_r(Q)} \big|\langle f,\phi_R \rangle\big|^2
 \leq
C\,\|f\|_{L^2(\mu)}^2,
\end{multline*}
and this shows that $\Sha_{r,s}$ is bounded on $L^2(\mu)$.

\begin{proof}[Proof of Theorem \ref{thm:HS-general}]
We first show that $\Xi(\Phi,\Psi; r,s)<\infty$ implies that $\Sha_{r,s}$  is of weak-type $(1,1)$. We shall assume that $\mu\in\mathcal{M}_\infty$ and the general case will be considered in Section \ref{section:general-mu} below. Let $\lambda>0$ be fixed and perform the Calder\'on-Zygmund decomposition in Theorem \ref{theor:CZ-d}.
Then,
\begin{align*}
\mu \{x \in \R^d: |\Sha_{r,s}  f(x)| > \lambda\} & \leq \mu\{ x\in \R^d : |\Sha_{r,s}  g(x)| >
\lambda/3\} + \mu(\Omega_\lambda) \\
&\quad\qquad + \mu\{ x\in \R^d\setminus \Omega_\lambda: |\Sha_{r,s}  b(x)| > \lambda/3\}  \\
&\quad\qquad+ \mu\{ x\in \R^d : |\Sha_{r,s}  \beta(x)| > \lambda/3\}\\
&= S_1 + S_2 +S_3 + S_4.
\end{align*}
Using the weak-type $(1,1)$ for $M_\dd$, Theorem \ref{theor:CZ-d} part $(a)$ and that $\Sha_{r,s}$ is bounded on $L^2(\mu)$ it is standard to check that
$$
S_1 + S_2  \leq \frac{C_{r,s}}{\lambda} \|f\|_{L^1(\mu)}.
$$
We next consider $S_3$. Let $x\in\R^d\setminus Q_j$ and observe that
\begin{multline}\label{eqn:Sha-bj}
|\Sha_{r,s} b_j(x)|
\le
\sup_{Q,R,S}|\alpha_{R,S}^Q|\sum_{Q\in\dd}\sum_{\begin{subarray}{c} R \in \dd _r(Q) \\ S \in \dd _s(Q) \end{subarray}} \big|\langle b_j,\phi_R \rangle\big| \,|\psi_S(x)|
\\
\lesssim
\sum_{Q_j\subsetneq Q\subset Q_j^{(r)}}\sum_{\begin{subarray}{c} R \in \dd _r(Q), R\subset Q_j \\ S \in \dd _s(Q) \end{subarray}} \big|\langle b_j,\phi_R \rangle\big| \,|\psi_S(x)|
.
\end{multline}
In the last inequality we have used that each non-vanishing term leads  to $Q_j\subsetneq Q\subset Q_j^{(r)}$ and $R\subset Q_j$ since $\phi_R$ is supported in $R$ and constant on the children of $R$, $b_j$ is supported in $Q_j$ and has vanishing integral, and $\psi_S$ is supported in $S$.  This, Chebyshev's inequality and Theorem \ref{theor:CZ-d} imply
\begin{multline*}
S_3
\le
\frac{3}{\lambda}\sum_j\int_{\R^d\setminus Q_j} |\Sha_{r,s}b_j|\,d\mu
\\
\lesssim
\frac{1}{\lambda}\sum_j\sum_{Q_j\subsetneq Q\subset Q_j^{(r)}}\sum_{\begin{subarray}{c} R \in \dd _r(Q), R\subset Q_j \\ S \in \dd _s(Q) \end{subarray}} \|b_j\|_{L^1(\mu)}\,\|\phi_R\|_{L^{\infty}(\mu)}\|\psi_S\|_{L^1(\mu)}
\\
\le
\frac{2^{(r+s)\,d}\,r}{\lambda}\, \Xi(\Phi,\Psi; r,s)\sum_j \|b_j\|_{L^1(\mu)}
\le
\frac{C_{r,s}}{\lambda}\,\|f\|_{L^1(\mu)}.
\end{multline*}

We finally estimate $S_4$. Let us observe that $\beta_j$ and $\phi_R$ have vanishing integral. Besides, $\beta_j$ is supported in $\widehat{Q}_j$ and constant on each dyadic child of $\widehat{Q}_j$, and $\phi_R$ is supported in $R$ and constant on each dyadic child of $R$. All these imply that $\langle \beta_j,\phi_R \rangle=0$ unless $R=\widehat{Q}_j$. Then,
\begin{multline}\label{eqn:Sha-betaj}
|\Sha_{r,s}\beta_j(x)|
\le
\sup_{Q,R,S}|\alpha_{R,S}^Q|\sum_{S\in\dd_s(Q_j^{(r+1)})} \big|\langle \beta_j,\phi_{\widehat{Q}_j} \rangle\big| \,|\psi_S(x)|
\\
\lesssim
\|\beta_j\|_{L^1(\mu)}\sum_{S\in\dd_s(Q_j^{(r+1)})} \|\phi_{\widehat{Q}_j}\|_{L^{\infty}(\mu)}\,|\psi_S(x)|.
\end{multline}
Therefore, Chebyshev's inequality and Theorem \ref{theor:CZ-d} imply
\begin{multline*}
S_4
\le
\frac{3}{\lambda}\sum_j\|\Sha_{r,s} \beta_j\|_{L^1(\mu)}
\\
\lesssim
\frac{1}{\lambda}\sum_j \|\beta_j\|_{L^1(\mu)}\!\!\!\sum_{S\in\dd_s(Q_j^{(r+1)})} \|\phi_{\widehat{Q}_j} \|_{L^\infty(\mu)}\|\psi_S\|_{L^1(\mu)}
\\
\le
\frac{2^{sd}}{\lambda}\Xi(\Phi,\Psi; r,s)
\sum_j \|\beta_j\|_{L^1(\mu)}
\le
\frac{C_{r,s}}{\lambda}\,\|f\|_{L^1(\mu)}.
\end{multline*}
Gathering the obtained estimates this part of the proof is complete.

We now turn to the converse, that is, we show that if a non-degenerate Haar shift $\Sha_{r,s}$ is of weak-type $(1,1)$ then $\Xi(\Phi,\Psi; r,s)<\infty$. For every $Q\in\dd_\Phi$, we pick $Q_\infty\in\dd_1(Q)$ such that $\phi_Q$ (which we recall that is constant on the dyadic children of $Q$) attains its maximum in $Q_\infty$. Define
$$
\widetilde{\varphi}_Q(x)
=
\big(\varphi_Q(x)-\langle \varphi_Q\rangle_Q\big)\,1_Q(x),
\quad
\varphi_Q(x)
=
\sgn\big(\phi_Q(x)\big)\frac{1_{Q_\infty}(x)}{\mu(Q_\infty)},
$$
where $\sgn(t)=t/|t|$ if $t\neq 0$ and $\sgn(0)=0$. We note that by construction $\widetilde{\varphi}_Q$  is supported on $Q$, constant on dyadic children of $Q$ and has vanishing integral. These imply that $\langle \widetilde{\varphi}_Q,\phi_{R}\rangle=0$ if $Q\neq R $. Also,
$$
\langle \widetilde{\varphi}_Q,\phi_Q\rangle
=
\langle \varphi_Q,\phi_Q\rangle
=
\frac{1}{\mu(Q_\infty)}\int_{Q_\infty} |\phi_Q(x)|\,d\mu(x)
=
\|\phi_Q\|_{L^\infty(\mu)},
$$
where we have used that $\phi_Q$ has vanishing integral and is constant on the dyadic children of $Q$. On the other hand,
$$
\|\widetilde{\varphi}_Q\|_{L^1(\mu)}
\le
2\int_{Q} |\varphi_Q(x)|\,d\mu(x)
=2.
$$

Let us now obtain that  $\Xi(\Phi,\Psi; r,s)<\infty$. In the definition of $\Xi(\Phi,\Psi; r,s)$ we may clearly assume that $R\in\dd_\Phi$ and $S\in\dd_\Psi$. Thus, we fix $Q_0\in\dd$, $R_0\in\dd_r(Q_0)$ and $S_0\in\dd_s(Q_0)$ with $\|\phi_{R_0}\|_{L^2(\mu)}=1$ and $\|\psi_{S_0}\|_{L^2(\mu)}=1$. We use the properties of the function $\widetilde{\varphi}_{R_0}$ just defined and the non-degeneracy of $\Sha_{r,s}$ to obtain that for every $x\in\R^d$
\begin{multline*}
|\Sha_{r,s} \widetilde{\varphi}_{R_0}(x)|
=
\Big|\sum_{Q \in \dd } \sum_{\begin{subarray}{c} R \in \dd _r(Q) \\ S \in \dd _s(Q) \end{subarray}} \alpha_{R,S}^Q\, \langle \widetilde{\varphi}_{R_0},\phi_{R}\rangle \psi_S(x)\Big|
\\
=
\|\phi_{R_0}\|_{L^\infty(\mu)}\,
\Big|\sum_{S \in \dd _s(Q_0) } \alpha_{R_0,S}^{Q_0}\psi_S(x)\Big|
\ge
\inf_{Q,R,S}|\alpha_{R,S}^{Q}| \,\|\phi_{R_0}\|_{L^\infty(\mu)}\, |\psi_{S_0}(x)|,
\end{multline*}
where we have used that $\dd _s(Q_0)$ is comprised of pairwise disjoint cubes. 
Using that $\Sha_{r,s}$ is of weak-type $(1,1)$ and that $\psi_{S_0}$ is constant on dyadic children of $S_0$, and \eqref{L1-weakL1} we obtain
\begin{multline*}
\|\phi_{R_0}\|_{L^\infty(\mu)}\,
\|\psi_{S_0}\|_{L^1(\mu)}
\approx
\big\|
\|\phi_{R_0}\|_{L^\infty(\mu)}\,\psi_{S_0}\big\|_{L^{1,\infty}(\mu)}
\\
\lesssim
\|\Sha_{r,s} \widetilde{\varphi}_{R_0}\|_{L^{1,\infty}(\mu)}
\lesssim
\|\widetilde{\varphi}_{R_0}\|_{L^1(\mu)}
\le 2.
\end{multline*}
This immediately implies that $\Xi(\Phi,\Psi; r,s)<\infty$.
\end{proof}

\begin{remark}\label{remark:HS-constant}
From the previous proof and a standard homogeneity argument on the parameter $\|\Sha_{r,s}\|_{L^2(\mu)\to L^{2}(\mu)}$ we obtain that, under the conditions of Theorem \ref{thm:HS-general},
\begin{multline*}
\|\Sha_{r,s}\|_{L^1(\mu)\to L^{1,\infty}(\mu)}\!
\le
C_0\,\big(\|\Sha_{r,s}\|_{L^2(\mu)\to L^{2}(\mu)}
\\
+
2^{s\,d}\,(r\,2^{r\,d}+1)\,\Xi(\Phi,\Psi; r,s)\sup_{Q,R,S} |\alpha_{R,S}^Q|\big),
\end{multline*}
where $C_0$ is a universal constant (independent of the dimension, for instance, in the previous argument one can safely take $C_0\le 217$.)
\end{remark}

\begin{remark}\label{remark:2-valued-Lp}
One can obtain an analog of Theorem \ref{thm:Hil} parts $(iii)$, $(iv)$ for non-degenerate Haar shift operators defined in terms of $2$-value Haar systems $\Phi$ and $\Psi$. To be more precise, let $\Sha_{r,s}$ be a non-degenerate Haar shift of complexity $(r,s)$ associated to two 2-value generalized Haar systems.  If $\Sha_{r,s}$ is of weak-type $(p,p)$ for some $1<p<2$ then $\Xi(\Phi,\Psi;r,s)<\infty$. The proof is very similar to what we did for the dyadic Hilbert transform. Fix $Q_0\in\dd$, $R_0\in\dd_r(Q_0)$, $S_0\in\dd_s(Q_0)$. Then, using that the cubes in $\dd _s(Q_0)$ are pairwise disjoint, 
\begin{multline*}
|\Sha_{r,s} \phi_{R_0}(x)|
=
\Big|\sum_{Q \in \dd } \sum_{\begin{subarray}{c} R \in \dd _r(Q) \\ S \in \dd _s(Q) \end{subarray}} \alpha_{R,S}^Q\, \delta_{R_0,R}\,\psi_S(x)\Big|
\\
=
\Big|\sum_{S \in \dd _s(Q_0) } \alpha_{R_0,S}^{Q_0}\psi_S(x)\Big|
\ge
\inf_{Q,R,S}|\alpha_{R,S}^{Q}|\, |\psi_{S_0}(x)|.
\end{multline*}
Using that $\Sha_{r,s}$ is of weak-type $(p,p)$ and that $\psi_{S_0}$ is constant on dyadic children of $S_0$ we obtain
$$
\|\psi_{S_0}\|_{L^p(\mu)}
\lesssim
\|\Sha_{r,s} \phi_{R_0}\|_{L^{p,\infty}(\mu)}
\lesssim
\|\phi_{R_0}\|_{L^p(\mu)}.
$$
Also, by \eqref{eqn:2-value-Haar}, \eqref{eqn:2-value-Haar:1-infty} and proceeding as in  \eqref{hI-Lp} we obtain
\begin{multline*}
\|\psi_{S_0}\|_{L^1(\mu)}^{1-\frac2{p'}}
\approx
m_\Psi(S_0)^{\frac12-\frac1{p'}}
\approx
\|\psi_{S_0}\|_{L^p(\mu)}
\\
\lesssim
\|\phi_{R_0}\|_{L^p(\mu)}
\approx
m_\Phi(R_0)^{\frac12-\frac1{p'}}
\approx
\|\phi_{R_0}\|_{L^\infty(\mu)}^{-(1-\frac2{p'})}.
\end{multline*}
This easily implies that $\Xi(\Phi,\Psi; r,s)<\infty$.
\end{remark}

\subsection{The case $\mu\in\mathcal{M}\setminus\mathcal{M}_\infty$}\label{section:general-mu}

The Calder\'on-Zygmund decomposition in Theorem \ref{theor:CZ-d} has been obtained under the assumption that every $d$-dimen\-sional quadrant has infinite $\mu$-measure, $\mu\in\mathcal{M}_\infty$ in the language of Section \ref{section:CZ}. Also, Theorems \ref{thm:Hil} and \ref{thm:HS-general} have been proved under this assumption. Here we discuss how to remove this constraint and work with arbitrary measures in $\mathcal{M}$.

Due to the nature of the standard dyadic grid, $\R^d$ splits naturally in $2^d$ components each of them being a $d$-dimensional quadrant.
Let $\R^d_k$, $1\le k\le 2^d$, denote the $d$-dimensional quadrants in $\R^d$: that is, the sets $\R^{\pm}\times\dots\times \R^{\pm}$ where $\R^+=[0,\infty)$ and $\R^-=(-\infty,0)$. Let $\dd^k$ be the collection of dyadic cubes contained in $\R^d_k$. We set
$$
M_{\dd^k} f(x) = \sup_{x\in Q\in \dd^k} \frac1{|Q|}\int_Q |f(y)| \,d\mu(y) = M_{\dd } \big(f\,1_{\mathbb{R}^d_k}\big)(x)\, 1_{\mathbb{R}^d_k}(x).
$$
Hence, given a function $f$  we have that
$$
f(x)=
\sum_{k=1}^{2^d} f(x)\,1_{\R^d_k}(x)
,\qquad
M_\dd f(x)
=
\sum_{k=1}^{2^d} M_{\dd^k}f(x)\,1_{\R^d_k}(x),
$$
and in each sum there is at most only one non-zero term. Because of this
decomposition, to extend our results it will suffice  to assume
that $f$ is supported in some $\R^d_k$ and obtain
the corresponding decompositions and estimates in $\R^d_k$.

Notice that if  $f$ is supported in $\R^d_k$, $M_\dd f=M_{\dd^k} f$ and this function is
supported in $\R^d_k$.   In particular, for any $\lambda>0$,
$$
\Omega_\lambda
=
\{x\in\R^d: M_\dd f(x)>\lambda\}
=
\{x\in\R^d_k: M_{\dd^k} f(x)>\lambda\},
$$
and so any decomposition of this set will consist of cubes in
$\dd^k$.    We modify our notation and define
$\langle f \rangle_{\R^d_k}=\frac1{\mu(\R^d_k)}\,\int_{\R^d_k}f\,d\mu$
if $\mu(\R^d_k)<\infty$ and
 $\langle f \rangle_{\R^d_k}=0$ if $\mu(\R^d_k)=\infty$.

The following result is the analog of Theorem \ref{theor:CZ-d}.
\begin{theorem}\label{theor:CZ-d:gral}
Given $1\le k\le 2^d$, $\mu\in\mathcal{M}$ and $f\in L^1(\mu)$ with $\supp f\subset\R^d_k$, so that for every $\lambda>\langle |f| \rangle_{\R^d_k}$ there exists a covering of $\Omega_\lambda = \{M_\dd  f > \lambda\}$ by maximal dyadic cubes $\{Q_j\}_j\subset\dd^k$. Then, we may find a decomposition $f = g + b + \beta$ with $g$, $b$ and $\beta$ as defined in Theorem \ref{theor:CZ-d} and satisfying the very same properties.
\end{theorem}

\begin{proof}
If $\mu(\R^d_k)=\infty$, then the proof given above goes through without change.
If $\mu(\R^d_k)<\infty$, then in the notation used above, $\langle |f|\rangle_Q\rightarrow \langle |f|\rangle_{\R^d_k}<\lambda$ as $\ell(Q)\to\infty$ for $Q\in\dd^k$.  Hence, if $Q\in \dd^k$ is such that $\langle |f|\rangle_Q>\lambda$, then $Q$ must
be contained in a maximal cube with the same property.  Hence, we can easily form the collection of maximal cubes $\{Q_j\}_j\subset\dd^k$. We observe that this covering gives the right estimate for the level sets of $M_\dd f= M_{\dd^k} f$
if $\lambda>\langle |f|\rangle_{\R^d_k}$. For $0<\lambda\le \langle |f|\rangle_{\R^d_k}$ we immediately have
$$
\mu(\Omega_\lambda)
\le
\mu(\R^d_k)
\le
\frac1\lambda\,\int_{\R^d_k} |f(x)|\,d\mu(x).
$$
These in turn imply that $M_{\dd^j}$ is of weak-type $(1,1)$. From here we repeat the arguments in the proof Theorem \ref{theor:CZ-d} to complete the proof without change.
\end{proof}

\begin{proof}[Proof of Theorems \ref{thm:Hil} and \ref{thm:HS-general} for $\mu\in\mathcal{M}$]
We obtain the weak-type $(1,1)$ estimate for $\Sha_{r,s}$, the arguments for $H_\dd$ and $H_\dd^*$  are identical.

Suppose first that  $\supp f\subset\R^d_k$ with $1\le k\le 2^d$. If $\mu(\R^d_k)=\infty$, then the arguments above go through without change.
Assume otherwise that $\mu(\R^d_k)<\infty$. If $\lambda>\langle |f| \rangle_{\R^d_k}$ then we repeat the same proof using Theorem \ref{theor:CZ-d:gral} in place of Theorem \ref{theor:CZ-d}. If $0<\lambda\le \langle |f| \rangle_{\R^d_k}$ we cannot form the Calder\'on-Zygmund decomposition. Nevertheless, the estimate is immediate after observing that by construction $\Sha_{r,s} f$ is supported in $\R^d_k$ since so is $f$. Then,
$$
\mu(\{x\in\R^d: |\Sha_{r,s} f(x)|> \lambda\})
\le
\mu(\R^d_k)
\le
\frac1{\lambda}\int_{\R^d_k}|f(x)|\,d\mu(x).
$$

To prove the weak-type estimate in the general case, fix $f$ and write $f=\sum_{k=1}^{2^d} f\, 1_{\R^d_k}$. By construction we then have
$$
\Sha_{r,s} f(x)=
\sum_{k=1}^{2^d} \Sha_{r,s} (f\,1_{\R^d_k})(x)\,1_{\R^d_k}(x).
$$
Therefore, by the above argument applied to each $\R^d_k$, we conclude as desired
\begin{multline*}
\mu(\{x\in\R^d: |\Sha_{r,s} f(x)|> \lambda\})
=
\sum_{k=1}^{2^d} \mu(\{x\in\R^d_k: |\Sha_{r,s} (f\,1_{\R^d_k})(x)|> \lambda\})
\\
\lesssim
\frac1{\lambda}\sum_{k=1}^{2^d}\int_{\R^d_k}|f(x)|\,d\mu(x)
=
\frac1{\lambda} \int_{\R^d}|f(x)|\,d\mu(x).
\end{multline*}
\end{proof}

\begin{remark}
As explained above, the standard dyadic grid splits $\R^d$  in $2^d$ components, each of them being a $d$-dimensional quadrant. These components are defined with respect to the property that if a given cube is in a fixed component, all of its relatives (ascendants and descendants) remain in the same component. This connectivity property depends on the dyadic grid chosen, and one can find other dyadic grids with other number of components. Let us work for simplicity in $\R$ and suppose that we want to find dyadic grids ``generated'' by $I_0=[0,1)$. We need to give the ascendants of $I_0$, say $I_{k}$, $k\le -1$. Once we have them, we translate each $I_k$ by $j\,2^{-k}$ with $j\in\Z$ and these define the cubes of the fixed generation $2^{-k}$ for $k\le 0$. The small cubes are obtained by subdivision. Hence, in the present scenario, we only need to define the $I_k$'s. Let us start by finding the parent of $I_0$: we just have two choices $[0,2)$ or $[-1,1)$, and once we choose one, which we call $I_{-1}$, we need to pass to the next level and decide which is the parent of $I_{-1}$, for which again we have two choices. Continuing this we have a sequence of cubes $I_k$, $k\le 0$, which determines the dyadic grid. In the classical dyadic grid one always choose the parent of $I_k$ ``to the right'', that is, so that $I_k$ is the left half of $I_{k-1}$. This eventually gives two components. One way to obtain a dyadic grid with one component is to alternatively take parents ``to the left'' and ``to the right''. That is, if we take $I_0=[0,1)$, $I_{-1}=[-1,1)$, $I_{-2}=[-1,3)$, $I_{-3}=[-5,3)$, \dots. we obtain one component. More precisely, take the family of intervals $I_k=[0, 2^{-k})$ for $k\ge 0$ and for $k\le -1$ let $I_k=[a_k-2^{-k},a_k)$ with $a_k=(2^{-k}+1)/3$ if $-k$ is odd and $a_k=(2^{-k+1}+1)/3$ if $-k$ is even. Notice that $\{I_k\}_{k\in\Z}$ is a decreasing family of intervals of dyadic side-length. Notice that each $I_k$ is one of the halves of $I_{k-1}$. Using $I_k$ we generate the dyadic cubes of generation $2^{-k}$ by taking the intervals $I_{j,k}=j\,2^{-k}+I_k$ with $j\in\Z$. Finally we set $\widetilde{\mathscr{D}}=\{I_{j,k}: j,k\in\Z\}$. This is clearly a dyadic grid in $\R$. Let us observe that $a_k\to \infty$ and $a_k-2^{-k}\to-\infty$ as $k\to -\infty$ and therefore $I_k\nearrow\R$ as $k\to-\infty$. This means that this dyadic grid induces  just one component (in the sense described above) since for any $I_1$, $I_2\in\widetilde{\mathscr{D}}$ we can find a large $k$ such that both $I_1$ and $I_2$ are contained in $I_{-k}\in \widetilde{\mathscr{D}}$.  We finally observe that the dyadic grids with one component occur more often than those with two,  as the classical dyadic grid. Indeed, if at each generation we select randomly the parent (among the possibilities ``to the left'' and ``to the right''), the probability of ending with a system with one component is $1$.
\end{remark}

\section{Examples of measures and Haar systems}
\label{section:examples}

\subsection{The $1$-dimensional case}\label{section:examples:1D}

As we have seen above the $1$-dimensional case is somehow special since the Haar system is ``uniquely'' determined. Let us work with the measures in Theorem \ref{thm:Hil}, that is, $\mu$ is a Borel measure in $\R$ with $0<\mu(I)<\infty$ for every $I\in\dd$. As we have seen in that result, $m$-increasing, $m$-decreasing and $m$-equilibrated measures are the ones governing the boundedness of $H_\dd$, $H_\dd^*$ and Haar shift operators. We are going to describe some examples of non-standard measures satisfying those conditions.

We can easily obtain examples of $m$-equilibrated measures. Let $\mu$ be a dyadically doubling measure, i.e., $\mu(\widehat{I}\,) \lesssim \mu(I)$ for all $I \in \dd$ where $\widehat{I}$ is the dyadic parent of $I$. Then, $m(I) \approx \mu(I)$ and clearly $\mu$ is $m$-equilibrated. This applies straightforwardly to the Lebesgue measure.

We next construct some measures that are $m$-increasing, $m$-decreasing or $m$-equilibrated without being dyadically doubling or of polynomial growth.
Set $d\nu = dx 1_{\mathbb{R}\setminus [0,1)} + d\mu$, where $\mu$ is a measure supported on the interval $[0,1)$ defined as follows. Let $\{I_k\}_{k \geq 0}$ be the decreasing sequence of dyadic intervals $I_k = [0,2^{-k})$ and let $\{a_k\}_{k \geq 1}$ be such that $0 < a_k< 1$ and $a_1 = 1/2$. Set $b_k=1-a_k$.  Define $\mu$ recursively by setting $\mu(I_0) = 1$ and
\begin{equation}
\label{eq:lasymm}
\mu(I_k) = a_k \mu(\widehat{I}_k)=a_k\mu(I_{k-1}) \quad \mbox{and} \quad \mu(I_k^b) = b_k\mu(\widehat{I}_k)=b_k\mu(I_{k-1})\,,
\end{equation}
for $k \geq 1$, where we recall that $I_k^b=[2^{-k}, 2^{-k+1})$ is the dyadic brother of $I_k$. On $I_k^b$, $\mu$ is taken to be uniform, i.e., $\mu(J) = \mu(I_k^b)\,|J|/|I_k^b|$ for any $J \in \dd$, $J\subset I_k^b$. We illustrate this procedure in Figure \ref{figure:mu}.

\begin{figure}
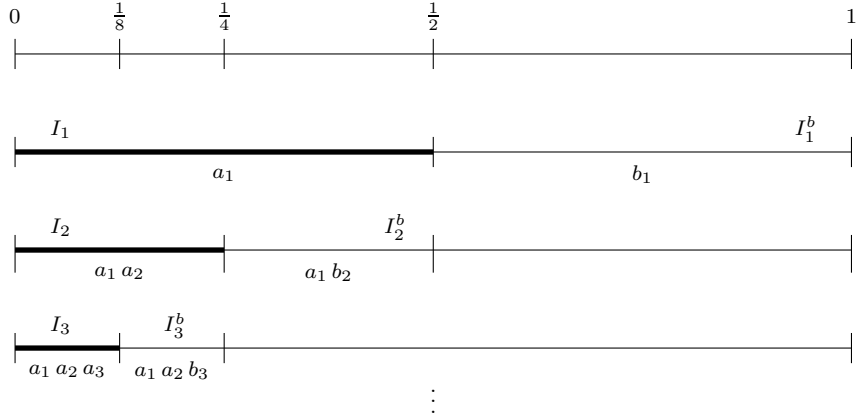

\begin{pgfpicture}{-6cm}{-6.5cm}{6cm}{0cm}


\pgfline{\pgfxy(-5.5,-1.5)}{\pgfxy(5.5,-1.5)}

\pgfline{\pgfxy(-5.5,-2.8)}{\pgfxy(5.5,-2.8)}

\pgfline{\pgfxy(-5.5,-4.1)}{\pgfxy(5.5,-4.1)}

\pgfline{\pgfxy(-5.5,-5.4)}{\pgfxy(5.5,-5.4)}


\pgfline{\pgfxy(-5.5,-1.7)}{\pgfxy(-5.5,-1.3)}
\pgfline{\pgfxy(-4.125,-1.7)}{\pgfxy(-4.125,-1.3)}
\pgfline{\pgfxy(-2.75,-1.7)}{\pgfxy(-2.75,-1.3)}
\pgfline{\pgfxy(0,-1.7)}{\pgfxy(0,-1.3)}
\pgfline{\pgfxy(5.5,-1.7)}{\pgfxy(5.5,-1.3)}


\pgfline{\pgfxy(-5.5,-3)}{\pgfxy(-5.5,-2.6)}
\pgfline{\pgfxy(0,-3)}{\pgfxy(0,-2.6)}
\pgfline{\pgfxy(5.5,-3)}{\pgfxy(5.5,-2.6)}


\pgfline{\pgfxy(-5.5,-4.4)}{\pgfxy(-5.5,-3.9)}
\pgfline{\pgfxy(-2.75,-4.4)}{\pgfxy(-2.75,-3.9)}
\pgfline{\pgfxy(0,-4.4)}{\pgfxy(0,-3.9)}
\pgfline{\pgfxy(5.5,-4.4)}{\pgfxy(5.5,-3.9)}


\pgfline{\pgfxy(-5.5,-5.6)}{\pgfxy(-5.5,-5.2)}
\pgfline{\pgfxy(-4.125,-5.6)}{\pgfxy(-4.125,-5.2)}
\pgfline{\pgfxy(-2.75,-5.6)}{\pgfxy(-2.75,-5.2)}
\pgfline{\pgfxy(5.5,-5.6)}{\pgfxy(5.5,-5.2)}


\pgfputat{\pgfxy(-5.5,-1)}{\pgfbox[center,center]{\scriptsize$0$}}
\pgfputat{\pgfxy(-4.125,-1)}{\pgfbox[center,center]{\scriptsize$\frac18$}}
\pgfputat{\pgfxy(-2.75,-1)}{\pgfbox[center,center]{\scriptsize$\frac14$}}
\pgfputat{\pgfxy(0,-1)}{\pgfbox[center,center]{\scriptsize$\frac12$}}
\pgfputat{\pgfxy(5.5,-1)}{\pgfbox[center,center]{\scriptsize$1$}}

\pgfputat{\pgfxy(-2.75,-3.1)}{\pgfbox[center,center]{\scriptsize$a_1$}}
\pgfputat{\pgfxy(2.75,-3.1)}{\pgfbox[center,center]{\scriptsize$b_1$}}

\pgfputat{\pgfxy(-4.9,-2.5)}{\pgfbox[center,center]{\scriptsize$I_1$}}
\pgfputat{\pgfxy(4.9,-2.5)}{\pgfbox[center,center]{\scriptsize$I_1^b$}}

\pgfputat{\pgfxy(-4.125,-4.4)}{\pgfbox[center,center]{\scriptsize$a_1\,a_2$}}
\pgfputat{\pgfxy(-1.375,-4.4)}{\pgfbox[center,center]{\scriptsize$a_1\,b_2$}}

\pgfputat{\pgfxy(-4.9,-3.8)}{\pgfbox[center,center]{\scriptsize$I_2$}}
\pgfputat{\pgfxy(-0.5,-3.8)}{\pgfbox[center,center]{\scriptsize$I_2^b$}}

\pgfputat{\pgfxy(-4.8125,-5.7)}{\pgfbox[center,center]{\scriptsize$a_1\,a_2\,a_3$}}
\pgfputat{\pgfxy(-3.4375,-5.7)}{\pgfbox[center,center]{\scriptsize$a_1\,a_2\,b_3$}}

\pgfputat{\pgfxy(-4.9,-5.1)}{\pgfbox[center,center]{\scriptsize$I_3$}}
\pgfputat{\pgfxy(-3.4,-5.1)}{\pgfbox[center,center]{\scriptsize$I_3^b$}}

\pgfputat{\pgfxy(0,-6)}{\pgfbox[center,center]{\scriptsize$\vdots$}}

\pgfsetlinewidth{2pt}

\pgfline{\pgfxy(-5.5,-2.8)}{\pgfxy(0,-2.8)}

\pgfline{\pgfxy(-5.5,-4.1)}{\pgfxy(-2.75,-4.1)}

\pgfline{\pgfxy(-5.5,-5.4)}{\pgfxy(-4.125,-5.4)}
\end{pgfpicture}
\\
\caption{Construction of $\mu$}\label{figure:mu}
\end{figure}

By construction, if  $I\cap I_0=\emptyset$ or $I_0\subset I$ we have
$$
\frac{m(I)}{m(\widehat{I}\,)}
=
\frac{|I|/4}{|\widehat{I}|/4}
=
\frac12
$$
Also, if  $I \in \dd$ and $\widehat{I}\subset I_k^b$ for some $k\ge 1$ then
$$
\frac{m(I)}{m(\widehat{I}\,)}
=
\frac
{\frac{\mu(I_k^b)\,|I|}{4\,|I_k^b|}}{\frac{\mu(I_k^b)\,|\widehat{I}|}{4\,|I_k^b|}}
=
\frac12
$$
In the remainder cases we always have that $\widehat{I}=\widehat{I}_k$ for some $k\ge 1$ and $I$ is either $I_k$ or $I_k^b$. Note that by \eqref{eq:lasymm} we get
\begin{align*}
m(I_k^b) &=
\frac{\mu\big((I_{k}^b)_-\big)\mu\big((I_{k}^b)_+\big)}{\mu(I_k^b)} = \frac{\mu(I_k^b)}{4}= \frac{1}{4}\,b_k\,\mu(\widehat{I}_k),
\\[4pt]
m(I_k)
&=
\frac{\mu\big((I_{k})_-\big)\mu\big((I_{k})_+\big)}{\mu(I_k)}
=
\frac{\mu(I_{k+1})\mu(I_{k+1}^b)}{\mu(I_k)}
=
a_{k+1}\,b_{k+1}\,a_k\,\mu(\widehat{I}_k),
\\[4pt]
m(\widehat{I}_k) &= \frac{\mu(I_k)\mu(I_k^b)}{\mu(\widehat{I}_k)} = a_k\, b_k\, \mu(\widehat{I}_k).
\end{align*}
Hence,
\begin{equation}
\label{eq:ems}
\frac{m(I_k)}{m(\widehat{I}_k)} = \frac{a_{k+1}b_{k+1}}{b_k} \qquad \textrm{and} \qquad \frac{m(I_k^b)}{m(\widehat{I}_k)} = \frac{1}{4a_k}.
\end{equation}

We now proceed to study the previous ratios associated to measures given by particular choices of the defining sequences $\{a_k\}_k$ and $\{b_k\}_k$. We shall construct three non-dyadically doubling and of non-polynomial growth  measures. In the first example $\mu$ is $m$-equilibrated, in the second $\mu$ is $m$-increasing and is not $m$-decreasing, in the third $\mu$ is $m$-decreasing and is not $m$-increasing. Finally, in the last example we give a measure $\mu$ which is of polynomial growth but is neither dyadically doubling, nor $m$-increasing, nor $m$-decreasing.

\begin{list}{$(\theenumi)$}{\usecounter{enumi}\leftmargin=.8cm
\labelwidth=.8cm\itemsep=0.2cm\topsep=.1cm
\renewcommand{\theenumi}{\alph{enumi}}}

\item  Let $b_k = \frac{1}{k}$ for $k\geq 2$. The measure $\mu$ is \textbf{non-dyadically doubling} since by \eqref{eq:lasymm}, if $k \ge 2$
$$
\frac{\mu(\widehat{I}_k)}{\mu(I_k^b)} = \frac1{b_k}=k \underset{k \to \infty}{\longrightarrow} \infty.
$$
From substituting  $a_k$ and $b_k$ in \eqref{eq:ems} we get that,
$$
\frac{m(I_k)}{m(\widehat{I}_k)} = \Bigl(1 - \frac{1}{k+1}\Bigr) \frac{k}{k+1}, \qquad\quad \frac{m(I_k^b)}{m(\widehat{I}_k)} = \frac{1}{4 \Bigl(1 - \frac{1}{k}\Bigr)}.
$$
Both sequences are bounded from above and from below, which implies that  $\mu$ is \textbf{$m$-equilibrated}. Besides, for $0 < t < \infty$
$$
\frac{\mu(I_k)}{|I_k|^t} = \frac{a_1\ldots a_k}{2^{-kt}} = \frac{1}{2}\frac{2^{kt}}{k}  \underset{k \to \infty}{\longrightarrow} \infty.
$$
Thus, $\mu$ does \textbf{not have polynomial growth}.

\item  Set $b_k = 2^{-k^2}$. In this case $\mu$ is \textbf{non-dyadically doubling}, since by \eqref{eq:lasymm}
$$
\frac{\mu(\widehat{I}_k)}{\mu(I_k^b)} =2^{k^2} \underset{k \to \infty}{\longrightarrow} \infty.
$$
Since $\frac{1}{2} \leq a_k < 1$, by \eqref{eq:ems} we get that $m(\widehat{I}_k) \approx m(I_k^b)$. However,
$$
4 < \frac{m(\widehat{I}_k)}{m(I_k)} = \frac{2^{-k^2}}{\bigl(1 - 2^{-(k + 1)^2}\bigr) 2^{- (k + 1)^2}} \underset{k \to \infty}{\longrightarrow} \infty.
$$
Thus, $\mu$ is \textbf{$m$-increasing} but is \textbf{not $m$-decreasing}. Notice that for $t > 1$,
$$
\frac{\mu(I_k)}{|I_k|^t} = \frac{a_1\ldots a_k}{2^{-kt}} = 2^{kt}\prod_{j=1}^k \bigl(1 - 2^{-j^2}\bigr) \geq 2^{kt}\Bigl(1 - \frac{1}{2}\Bigr)^k = 2^{k(t-1)}  \underset{k \to \infty}{\longrightarrow} \infty.
$$
For $0 < t \leq 1$, let $n$ and $m$ be positive integers such that $\frac{1}{n + 1} < t \leq \frac{1}{n}$ and $k = 2(n + 1)m$. Then, $2^{kt} > 2^{2m}$ and
\begin{multline*}
\frac{\mu(I_k)}{|I_k|^t}
\ge \Big(2^m\prod_{j=1}^m \bigl(1 - 2^{-j^2}\bigr)\Big)\cdot  \Big(2^m\prod_{j=m+1}^k \bigl(1 - 2^{-j^2}\bigr)\Big)
\ge
2^m\prod_{j=m+1}^k \bigl(1 - 2^{-j^2}\bigr)
\\
\ge
2^m\bigl(1 - 2^{-m^2}\bigr)^{k - m}
= \Bigl(2\bigl(1 - 2^{-m^2}\bigr)^{(2(n+1) - 1)}\Bigr)^m\underset{m \to \infty}{\longrightarrow} \infty
\end{multline*}
Thus, $\mu$ does \textbf{not have polynomial growth}.

\item Let $n \in \mathbb{N}$ and set $f(n) = \frac{n(n+1)}{2}$.
For $k \geq 2$ define
$$
b_k = \frac{1}{2}\,\frac{1}{k - f(n-1)},
$$
where $n\ge 2$ is such that $f(n-1) < k \leq f(n)$. Fix $n \geq 2$ and $f(n-1) < k \leq f(n)$. Then $k = f(n-1) + r$, with $1 \leq r \leq f(n) - f(n-1) = n$ and $b_k = 1/(2r)$. Hence,
$$
\frac{1}{2n} \leq b_k \leq \frac{1}{2}.
$$
and $\liminf_{k \to \infty} b_k = 0$. By \eqref{eq:lasymm} this choice of $b_k$ defines a \textbf{non-doubling} measure. Since $\frac{1}{2} \leq a_k < 1$, by \eqref{eq:ems} we get that $m(\widehat{I}_k) \approx m(I_k^b)$ for every $k$. On the other hand,
$$
\frac{b_{k+1}}{b_k}
=
\begin{cases}
\displaystyle \frac{k - f(n-1)}{k+1 - f(n-1)} = \frac{r}{r+1} \approx 1, & \textrm{if $k < f(n)$;} \\[12pt]
\displaystyle \frac{k - f(n-1)}{k+1 - f(n)} = n \to \infty, & \textrm{if $k = f(n)$.}
\end{cases}
$$
Hence, by \eqref{eq:ems} $\mu$ is \textbf{not $m$-increasing}. However, $\mu$ is \textbf{$m$-decreasing} since ${b_k}/{b_{k+1}} \leq 2$.

We finally see that $\mu$ has no polynomial growth. We start with the case $t > 1$. For $s,j \geq 2$ such that $f(s-1) < j = f(s-1) + r \leq f(s)$  with $1\le r\le s$, we have that $a_j = \frac{2r-1}{2r}$. Then, if $k=f(n)$
\begin{multline*}
\frac{\mu(I_k)}{|I_k|^t} = \frac{a_1\ldots a_k}{2^{-kt}}
=
2^{kt}
\prod_{s=1}^n \prod_{r=1}^s \frac{2r-1}{2r}
=
2^{k(t-1)} \prod_{s=1}^n \prod_{r=1}^s \frac{2r-1}{r} \geq 2^{k(t-1)}  \\
=
2^{f(n)\,(t-1)}
\underset{n \to \infty}{\longrightarrow} \infty.
\end{multline*}
Consider now $0 < t \leq 1$ and let $m \geq 2$ be the unique integer such that $\frac{2}{f(m)} < t \le \frac{2}{f(m-1)}$. Let $k=f(n)$ with $n$ large enough  so that $k \geq f(m)^2$. Then $2^{kt} \geq 2^{2f(m)}$ and
\begin{align*}
\frac{\mu(I_k)}{|I_k|^t} &
\ge
\Big(2^{f(m)}\prod_{s=1}^m \prod_{r=1}^s \frac{2r-1}{2r}\Big)
\cdot
\Big(2^{f(m)}\prod_{s=m+1}^n \prod_{r=1}^s \frac{2r-1}{2r}\Big)
\\
&\geq 2^{f(m)}\prod_{s=m+1}^n \prod_{r=1}^s \frac{2r-1}{2r} \\
&= 2^{f(m)}\prod_{s=m+1}^n \frac{(2s)!}{2^{2s}(s!)^2} \\
&= 2^{f(m)}2^{-2(f(n) - f(m))}\prod_{s=m+1}^n \frac{(2s)!}{(s!)^2} \\
&\geq 2^{f(m)}2^{-2(f(n) - f(m))}2^{3(f(n) - f(m))}
= 2^{f(n)}\underset{n \to \infty}{\longrightarrow} \infty,
\end{align*}
where in the last inequality we have used that $(2s)!/(s!)^2$ is increasing and therefore bounded from below by 8.
Thus, $\mu$ does \textbf{not have polynomial growth}.

\item Let $b_2=b_3=1/2$, and for every $k\ge 2$, $b_{2k}=1/k$, $b_{2k+1}=1-1/k$. The measure $\mu$ is \textbf{non-dyadically doubling} since by \eqref{eq:lasymm}, if $k \ge 2$, then
$$
\frac{\mu(\widehat{I}_{2k})}{\mu(I_{2k}^b)} = \frac1{b_{2k}}=k \underset{k \to \infty}{\longrightarrow} \infty.
$$
From substituting  $a_k$ and $b_k$ in \eqref{eq:ems} we get that,
$$
\frac{m(I_{2k+1}^b)}{m(\widehat{I}_{{2k+1}})} = \frac{1}{4 a_{2k+1}}=\frac{k}{4}\underset{k \to \infty}{\longrightarrow} \infty,
$$
which implies that $\mu$ is \textbf{not $m$-increasing}. Also,
$$
\frac{m(\widehat{I}_{{2k+1}})}{m(I_{2k+1})} = \frac{b_{2k+1}}{a_{2(k+1)}\,b_{2(k+1)}}=\frac{(k+1)^2\,(k-1)}{k^2}\underset{k \to \infty}{\longrightarrow} \infty,
$$
which implies that $\mu$ is \textbf{not $m$-decreasing}.

We finally see that \textbf{$\mu$ has linear growth}, that is, $\mu(I)/|I|\le C$ for every $I$. We first notice that it suffices to consider $I\in\dd$ since  any arbitrary interval $J$ can be covered by a bounded number of $I\in\dd$ with $|I|\approx|J|$. Let us now fix $I\in\dd$. The cases $I\cap[0,1)=\emptyset$ or $[0,1)\subset I$ are trivial since $\mu(I)=|I|$. Suppose next that $I\subsetneq [0,1)$. Then, either $I=I_k$ or $I\subset I_k^b$ for some $k\ge 1$. In the latter scenario we have that by construction $\mu(I)/|I|=\mu(I_k^b)/|I_k^b|$, therefore we only have to consider  $I=I_k$ or $I=I_k^b$ for $k$ large. Let us fix $k\ge 6$. Notice that
$$
\frac{\mu(I_k^b)}{|I_k^b|}
=
\frac{\mu(I_k)}{|I_k|}\,\frac{b_k}{a_k}
=
\frac{\mu(I_{k-1})}{|I_{k-1}|}\,2\,b_k.
$$
Thus,
$$
\frac{\mu(I_{2k}^b)}{|I_{2k}^b|}
=
\frac{\mu(I_{2k})}{|I_{2k}|}\,\frac{b_{2k}}{a_{2k}}
\le
\frac{\mu(I_{2k})}{|I_{2k}|},
\qquad
\frac{\mu(I_{2k+1}^b)}{|I_{2k+1}^b|}
=
\frac{\mu(I_{2k})}{|I_{2k}|}\,2\,b_{2k+1}
\le
2
\frac{\mu(I_{2k})}{|I_{2k}|}.
$$
Additionally,
$$
\frac{\mu(I_{2k+1})}{|I_{2k+1}|}
=
\frac{\mu(I_{2k})}{|I_{2k}|}\,2\,a_{2k+1}
\le
2\,\frac{\mu(I_{2k})}{|I_{2k}|}.
$$
All these together show that it suffices to bound $\mu(I_{2k})/|I_{2k}|$ for $k\ge 3$. Let $k\ge 3$, then we obtain  as desired
$$
\mu(I_{2k})
=
\prod_{j=1}^{2k}a_j
=
2^{-3}\Big(\prod_{j=2}^k a_{2j}\Big)\,\Big(\prod_{j=2}^{k-1} a_{2j+1}\Big)
=
2^{-3}\frac1{k!}\le
\frac43\, 2^{-2k}
=
\frac43\, |I_{2k}|.
$$

\end{list}

\subsection{The higher dimensional case: specific Haar system constructions}\label{section-Haar-systems}

As we have shown in Theorem \ref{thm:HS-general}, the weak-type $(1,1)$ estimate for Haar shifts is governed by the finiteness of the quantities $\Xi(\Phi, \Psi;r,s)$.
In the $1$-dimensional case, these can be written only in terms of the measure $\mu$ since the Haar system $\mathcal{H}$ is ``unique'' (see Remark \ref{remarks:def-GHS}). However in higher dimensions
we have different choices of the Haar system and each of them may lead to a different condition. Therefore, before getting into that let us construct some specific Haar systems.


Among the $\mu$-Haar systems in higher dimensions, two of them are relatively easy to construct: Wilson's Haar system and Mitrea's Haar system \cite{wilson2002}, \cite{dragicevic-petermichle-volberg2006}, \cite{chung2011}, \cite{mitrea1994}, \cite{hytonen2012a}. Following \cite{girardi-sweldens1997}, we present a simplified way of obtaining this two $\mu$-Haar systems for measures $\mu \in \mathcal{M}$.

To construct Wilson's Haar system, start with some enumeration $(Q_j)_{j=1}^{2^d}$ of the dyadic children of $Q$ and build a dyadic (or logarithmic) partition tree on it. The partition is given as follows: set $\mathscr{W}_0(Q) = \{\{1, 2, \dots, 2^d \}\}$ and let $\mathscr{W}_1(Q) = \{ \{1,\dots, 2^{d-1}\},\{2^{d-1}+1,\dots, 2^d\}\}$. Proceed recursively to get the partition $\mathscr{W}_{k}(Q)$, obtained upon halving the elements of $\mathscr{W}_{k-1}(Q)$ and ending up with $\mathscr{W}_d(Q) = \{\{1\}, \{2\}, \ldots, \{2^d \}\}$. Set
$$
E_Q^\omega = \bigcup_{j \in \omega} Q_j
\qquad\qquad\mbox{with}\qquad\qquad
\omega\in \mathscr{W}(Q) = \bigcup_{k=0}^{d-1} \mathscr{W}_{k}(Q).
$$

We are going to see that the family of sets $\{E_Q^{\omega}\}_{\omega\in \mathscr{W}(Q)}$ behaves like a $1$-dimensional dyadic grid. Form construction, any $\omega \in \mathscr{W}_{k-1}(Q)$, $1 \leq k \leq d$,  has two disjoint children $\omega_-, \omega_+ \in \mathscr{W}_{k}(Q)$ such that $\omega = \omega_- \cup\, \omega_+$. Thus, following the notation of the $1$-dimensional case, we write $(E_Q^\omega)_-=E_Q^{\omega_-}$ and $(E_Q^\omega)_+=E_Q^{\omega_+}$. Note that these two sets are disjoint and $E_Q^\omega = (E_Q^{\omega})_- \cup (E_Q^{\omega})_+$. We call $(E_Q^{\omega})_-$ and $(E_Q^{\omega})_+$ the dyadic children of $E_Q^\omega$.
Besides, for every $\omega \in \mathscr{W}_k(Q)$, $1\le k\le d$, there exists a unique $\widehat{\omega} \in \mathscr{W}_{k-1}(Q)$ such that $\widehat{\omega} \supset \omega$ and thus $E_Q^\omega\subset \widehat{E_Q^\omega}=E_Q^{\widehat{\omega}}$. We call $\widehat{E_Q^\omega}$ the dyadic parent of $E_Q^\omega$.  Moreover, $E_Q^{\omega}$ and $E_Q^{\omega'}$ are either disjoint or one is contained in the other.

We define the Haar functions adapted to the family of sets $\{E_Q^{\omega}\}_{\omega\in \mathscr{W}(Q)}$: for every $\omega\in \mathscr{W}(Q)$ we set
$$
h_Q^\omega = \sqrt{m(E_Q^w)}\Biggl(\frac{1_{(E_Q^\omega)_-}}{\mu((E_Q^\omega)_-)} - \frac{1_{(E_Q^\omega)_+}}{\mu((E_Q^\omega)_+)}\Biggr),
$$
where
\begin{multline*}
m(E_Q^\omega)
=
\frac{\mu((E_Q^\omega)_-)\,\mu((E_Q^\omega)_+)}{\mu(E_Q^\omega)}
=
\Biggl(\frac{1}{\mu((E_Q^\omega)_-)} + \frac{1}{\mu((E_Q^\omega)_+)}\Biggr)^{-1}
\\
\approx
\min\big\{\mu((E_Q^\omega)_-), \mu((E_Q^\omega)_+) \big\}.
\end{multline*}
Note that this makes sense provided $\mu((E_Q^\omega)_-)\,\mu((E_Q^\omega)_+)>0$. For otherwise, we set $h_Q^\omega\equiv 0$.

Note that for a fixed $Q\in\dd$  and  $\omega\in \mathscr{W}(Q)$, one can easily verify that $h_Q^\omega$ satisfies the properties $(a)$--(d) in Definition \ref{def:GHS}. Let us further observe that $h_Q^\omega$ is orthogonal to $h_Q^{\omega'}$ for $\omega\neq\omega'$. We would like to emphasize that here we have $2^{d}-1$ generalized Haar functions associated to each $Q$ (one for each $\omega\in \mathscr{W}(Q)$). In this way, if for every $Q$ we pick $\omega_Q\in\mathscr{W}(Q)$, we have that $\{h_{Q}^{\omega_Q}\}_{Q\in\dd}$ is a $2$-value generalized Haar system in $\R^d$ (see Definition \ref{def:GHS} and Remark \ref{remarks:def-GHS}) and therefore standard (see \eqref{def:standard}).

\smallskip

Mitrea's Haar system is constructed in the following way. Let us fix an enumeration $(Q_j)_{j=1}^{2^d}$ of the dyadic children of $Q$. For every $2\le j \le 2^d$ we set $\widetilde{Q}_j=\cup_{k=j}^{2^d} Q_k$. We define Mitrea's Haar system as follows: for every $1\le j\le 2^{d}-1$ we set
$$
H_Q^j  = \sqrt{m(Q_j)}\Biggl(\frac{1_{Q_j}}{\mu(Q_j)} - \frac{1_{\widetilde{Q}_{j+1}}}{\mu(\widetilde{Q}_{j+1})}\Biggr),
$$
where
$$
m(Q_j)
=
\frac{\mu(Q_j)\,\mu(\widetilde{Q}_{j+1})}{\mu(\widetilde{Q}_{j})}
=
\Biggl(
\frac{1}{\mu(Q_j)} + \frac{1}{\mu(\widetilde{Q}_{j+1})}\Biggr)^{-1}
\!\approx \min\{\mu(Q_j),\mu(\widetilde{Q}_{j+1})\}.
$$
This definition makes sense provided $\mu(Q_j)\,\mu(\widetilde{Q}_{j+1})>0$. For otherwise, we set $H_Q^j\equiv 0$.

Again, for a fixed $1\le j\le 2^d-1$ and  $Q\in\dd$, one can easily verify that $H_Q^j$ satisfies the properties $(a)$--$(d)$ in Definition \ref{def:GHS} and also that $H_Q^j$ is orthogonal to $H_Q^{j'}$ for $j\neq j'$. As before, we have $2^{d}-1$ generalized Haar functions associated to each $Q$ (one for each $j$). Hence,  if for every $Q$ we pick $j_Q$, $1\le j_Q\le 2^d-1$, we have that $\{H_{Q}^{j_Q}\}_{Q\in\dd}$ is a $2$-value generalized Haar system in $\R^d$ (see Definition \ref{def:GHS} and Remark \ref{remarks:def-GHS}) and therefore standard (see \eqref{def:standard}).

We finally present another way to construct Haar systems in the spirit of the wavelet construction. For this example, we assume that $\mu$ is a product measure, that is, $\mu=\mu_1\times\dots\times\mu_d$ where $\mu_1,\dots,\mu_d$ are Borel measures in $\R$ satisfying $\mu_j(I)<\infty$ for every $I\in\dd$. We will use the following notation, given $Q\in\dd(\R^d)$ we have that $Q=I_1^Q\times\dots\times I_d^Q$ with $I_j^Q\in\dd(\R)$. Hence, $\mu(Q)=\prod_{j=1}^d \mu_j(I_j^Q)$. Associated to each $\mu_j$ we consider a $\mu_j$-generalized Haar system $\Phi_{j}=\{\phi_{j,I}^1\}_{I\in\dd(\R)}$. For every $I\in\dd(\R)$ with $\mu_j(I)>0$ we set $\phi_{j, I}^{0}=1_{I}/\mu_j(I)^{\frac12}$ and $\phi_{j, I}^{0}\equiv 0$  otherwise. For every $\epsilon=(\epsilon_1,\dots,\epsilon_d)\in\{0,1\}^d\setminus\{0\}^d$ and $Q\in\dd(\R^d)$ we define
$$
\phi_Q^\epsilon(x)
=
\prod_{j=1}^d \phi_{j, I_j^Q}^{\epsilon_j}(x_j).
$$
We have that each $\phi_Q^\epsilon$ satisfies the properties $(a)$--$(d)$ in Definition \ref{def:GHS} and also that $\phi_Q^\epsilon$ is orthogonal to $\phi_Q^{\epsilon'}$ for $\epsilon\neq \epsilon'$. Hence,  if for every $Q$ we pick $\epsilon_Q$,  as above, we have that $\{\phi_{Q}^{\epsilon_Q}\}_{Q\in\dd}$ is a generalized Haar system in $\R^d$, see Definition \ref{def:GHS}. Note that Remark \ref{remarks:def-GHS} says each $\Phi_j$ is a $2$-value generalized Haar system in $\R$. However, unless some further condition is imposed in each measure $\mu_j$, one has that $\phi_Q^\epsilon$ may take more than 2 non-vanishing values (this is quite easy if we take $\epsilon=\{1\}^d$). Nevertheless, if $Q\in\dd_\Phi$ then
$$
\|\phi_Q^\epsilon\|_{L^1(\mu)}
=
\prod_{j=1}^d \|\phi_{j, I_j^Q}^{\epsilon_j}\|_{L^1(\mu_j)},
\qquad
\|\phi_Q^\epsilon\|_{L^\infty(\mu)}
=
\prod_{j=1}^d \|\phi_{j, I_j^Q}^{\epsilon_j}\|_{L^\infty(\mu_j)}.
$$
Let $m_j(I)=\mu_j(I_-)\mu_j(I_+)/\mu_j(I)$ for $I\in\dd_{\Phi_j}$. Then we have that, for every $I\in\dd_{\Phi_j}$,
$$
\|\phi_{j, I}^{0}\|_{L^1(\mu_j)}=\sqrt{\mu_j(I)},\qquad \quad
\|\phi_{j, I}^{0}\|_{L^\infty(\mu_j)}=\frac1{\sqrt{\mu_j(I)}},
$$
and, as in \ref{hI:1-infty},
$$
\|\phi_{j, I}^{1}\|_{L^1(\mu_j)}=2\sqrt{m_j(I)},\qquad\quad
\|\phi_{j, I}^{1}\|_{L^\infty(\mu_j)}\approx \frac1{\sqrt{m_j(I)}}.
$$
Thus, despite the fact that $\Phi$ is not a $2$-value generalized Haar system in general, we obtain that $\Phi$ is standard.

To conclude this section we observe that although the generalized Haar systems we have constructed above are all standard, this is not the case in general. We work in $\R^2$ and for $k\ge 2$ we let
$Q_k=[k,k+1)\times [k,k+1)$. Fix an enumeration $Q_k^1, Q_k^2, Q_k^3, Q_k^4$ of the dyadic children of $Q_k$. Define $F(x)\equiv 1$ if $x\notin \cup_{k\ge 2}Q_k$ and elsewhere
$$
F(x)=
\sum_{k=2}^\infty \Big(
\frac4{k^2}\,\big(1_{Q_k^1}(x)+1_{Q_k^2}(x)\big)
+
\frac{2(k^2-2)}{k^2}\,\big(1_{Q_k^3}(x)+\,1_{Q_k^4}(x)\big)\Big).
$$
We consider $d\mu(x)=F(x)\,dx$ which is a Borel measure such that $0<\mu(Q)<\infty$ for every $Q\in\dd$. By construction we have
$$
\mu(Q_k^1)=\mu(Q_k^2)=\frac1{k^2},
\qquad
\mu(Q_k^3)=\mu(Q_k^4)=\frac{k^2-2}{2k^2},
\qquad
\mu(Q_k)=1.
$$
Next we consider the system $\Phi=\{\phi_{Q_k}\}_{k\ge 2}$ with
\begin{multline*}
\phi_{Q_k}
=
\frac1{2k}\,\Big(\frac{1_{Q_k^1}}{\mu(Q_k^1)}-\frac{1_{Q_k^2}}{\mu(Q_k^2)}\Big)
+
\sqrt{\frac{k^2-2}{8\,k^2}}
\,\Big(\frac{1_{Q_k^3}}{\mu(Q_k^3)}-\frac{1_{Q_k^4}}{\mu(Q_k^4)}\Big)
\\
=
\frac{k}{2}\,\big(1_{Q_k^1}-1_{Q_k^2}\big)
+
\sqrt{\frac{k^2}{2\,(k^2-2)}}
\,\big(1_{Q_k^3}-1_{Q_k^4}\big).
\end{multline*}
By construction each $\phi_{Q_k}$ satisfies $(a)$--$(d)$ in Definition \ref{def:GHS} where we observe that in $(d)$ we have $\|\phi_{Q_k}\|_{L^2(\mu)}=1$. Thus, $\Phi$ is a generalized Haar system in $\R^2$. On the other hand,
\begin{multline*}
\|\phi_{Q_k}\|_{L^1(\mu)}\|\phi_{Q_k}\|_{L^\infty(\mu)}
=
\Bigg(\frac1k+\sqrt{\frac{k^2-2}{2\,k^2}}\Bigg)
\max
\Bigg\{
\frac{k}{2},\sqrt{\frac{k^2}{2\,(k^2-2)}}\,
\Bigg\}
\\
\ge
\sqrt{\frac{k^2-2}{2\,k^2}}\,\frac{k}{2}
=
\frac{\sqrt{k^2-2}}{2\,\sqrt{2}}\underset{k \to \infty}{\longrightarrow} \infty.
\end{multline*}
Therefore, $\Phi$ is not standard. We note that in view of Example \ref{ex:Haar-mult} we have that the Haar multiplier
\begin{equation}\label{HM-non-stan}
T_\epsilon f
=
\sum_{Q\in\dd} \epsilon_Q\,\langle f,\phi_{Q}\rangle\,\phi_{Q},
\qquad
\epsilon_Q=\pm 1
\end{equation}
is not of weak-type $(1,1)$. We can obtain this from Theorem \ref{thm:HS-general}. However, here the situation is very simple: we just take $\varphi_{Q_k}=1_{Q_k^1}/\mu(Q_k^1)$ and obtain that
$$
T_\epsilon \varphi_{Q_k}
=
\epsilon_{Q_k}\,\langle \varphi_{Q_k},\phi_{Q_k}\rangle\,\phi_{Q_k}
=
\epsilon_{Q_k}\,\frac{k}2\,\phi_{Q_k}.
$$
Thus, by \eqref{L1-weakL1},
$$
\frac{\|T_\epsilon \varphi_{Q_k}\|_{L^{1,\infty}(\mu)}}{\|\varphi_{Q_k}\|_{L^1(\mu)}}
\approx
\frac{k\,\|\phi_{Q_k}\|_{L^{1}(\mu)}}{2}
=
\frac{k}{2}\,\Big(\frac1k+\sqrt{\frac{k^2-2}{2\,k^2}}\Big)\underset{k \to \infty}{\longrightarrow} \infty,
$$
and therefore $T_\epsilon$ is not of weak-type $(1,1)$.

Let us finally point out that in the classical situation (i.e., when $\mu$ is the Lebesgue measure and we take a standard Haar system) these operators are usually referred to as a martingale transforms. As it is well known, martingale transforms are of weak-type $(1,1)$ for any measure $\mu$ by the use of  probability methods. Surprisingly, $T_\epsilon$ is not of weak-type $(1,1)$ and therefore $T_\epsilon$ cannot be written as a ``martingale transform'' operator in terms of martingale differences (see \eqref{eqn:martin-transform-hI} below for further details).

\subsection{Examples of measures in higher dimensions}

Taking into account the previous constructions, we are going to give some examples of non trivial measures so that the conditions in Theorem \ref{thm:HS-general} hold. We first notice that if $\mu$ is dyadically doubling then $\mu(Q)\approx \mu(Q')$ for every dyadic children $Q'$ of $Q$. In particular, for any generalized Haar system $\Phi$, one can show that
$\|\Phi_Q\|_{L^1(\mu)}\approx \mu(Q)^{1/2}$ and $\|\Phi_Q\|_{L^\infty(\mu)}\approx \mu(Q)^{-1/2}$ for every $Q\in\mathscr{D}_\Phi$. This clearly implies that we always have that $\Xi(\Phi,\Psi;r,s)\le C_{r,s}$ for any choices of generalized Haar systems.
Thus, the problem becomes interesting when $\mu$ is not dyadically doubling. The general case admits too many choices, and we just want to give an illustration of the kind of issues that one can find. Therefore we are going to restrict ourselves to dimension $d=2$ with $0<\mu(Q)<\infty$ for every $Q\in\dd(\R^2)$ and $\Phi=\Psi$ with $\dd_\Phi=\dd$. We are going to consider the complexities $(1,0)$ and $(0,1)$ (since these are related to the model operators $H_\dd$ and $H_\dd^*$ in $1$-dimension).

We consider Wilson's construction. We halve each $Q$ horizontally and write $Q_N$ for the northern ``hemisphere'' and $Q_S$ the southern ``hemisphere''. If for every cube $Q$ we take the anti-clockwise enumeration starting with the west-south corner then $Q_S=E_Q^{\{1,2\}}$ and $Q_N=E_Q^{\{3,4\}}$. We now take Wilson's system $\Phi=\{h_Q^{\{1,2,3,4\}}\}_{Q\in\dd}$, that is,
$$
h_Q^{\{1,2,3,4\}}
=
\sqrt{m_{N,S}(Q)}\,\Big(\frac{1_{Q_S}}{\mu(Q_S)}-\frac{1_{Q_N}}{\mu(Q_N)}\Big),
\qquad
m_{N,S}(Q)=\frac{\mu(Q_S)\,\mu(Q_N)}{\mu(Q)}.
$$

Suppose that $d\mu(x,y)=dx\,d\nu(y)$ then $\mu$ is dyadically doubling iff $\nu$ is dyadically doubling.
If $Q=I\times J$ then
$$
m_{N,S}(Q)
=
|I|\,m_{\nu}(J)
=
|I|\,\frac{\nu(J_-)\,\nu(J_+)}{\nu(J)}.
$$
Then $\Xi(\Phi,\Phi;0,1)<\infty$ if and only if $\nu$ is $m_\nu$-increasing and $\Xi(\Phi,\Phi;1,0)<\infty$ if and only if $\nu$ is $m_\nu$-decreasing. Using the examples we constructed above we find measures $\mu$ in $\R^2$ which are non-dyadically doubling but they satisfy one (or both) conditions.

However if we use another Haar system we get a different behavior. Suppose now that our enumeration is clockwise  and starts with the west-south corner then $Q_W=E_Q^{\{1,2\}}$ and $Q_E=E_Q^{\{3,4\}}$ are respectively the western and eastern ``hemispheres''. If now take Wilson's system $\Phi=\{h_Q^{\{1,2,3,4\}}\}$ then we get the same definitions as before replacing $Q_S$ by $Q_W$ and $Q_N$ by $Q_E$. In particular,
$$
m_{E,W}(Q)
=
\frac{|I|}{4}\,\nu(J)
$$
Then we always have $\Xi(\Phi,\Phi;0,1)\le 1/\sqrt{2}<\infty$, whereas $\Xi(\Phi,\Phi;1,0)<\infty$ if and only if $\nu$ is dyadically doubling.

Similar examples can be constructed using Mitrea's Haar shifts.

We finally look at the Haar system using the wavelet construction. If our system is comprised of $\phi_Q^{i,1}(x,y)=\phi_{1,I}^i(x)\phi_{2,J}^1(y)$ with $i=0$ or $1$ we obtain
$$
\|\phi_{Q}^{i,1}\|_{L^1(dx\times d\nu)}
=
2\,\sqrt{|I|\,m_\nu(J)},
\qquad
\|\phi_{Q}^{i,1}\|_{L^\infty(dx\times d\nu)}
\approx
\frac1{\sqrt{|I|\,m_\nu(J)}},
$$
and then we have the same behavior as before: $\Xi(\Phi,\Phi;0,1)<\infty$ if and only if $\nu$ is $m_\nu$-increasing and $\Xi(\Phi,\Phi;1,0)<\infty$ if and only if $\nu$ is $m_\nu$-decreasing.
On the other hand, if we take $\phi_Q^{1,0}(x,y)=\phi_{1,I}^1(x)\phi_{2,J}^0(y)$ and obtain
$$
\|\phi_{Q}^{1,0}\|_{L^1(dx\times d\nu)}
=
\sqrt{|I|\,\nu(J)},
\qquad
\|\phi_{Q}^{1,0}\|_{L^\infty(dx\times d\nu)}
=
\frac1{\sqrt{|I|\,\nu(J)}}.
$$
Then we always have $\Xi(\Phi,\Phi;0,1)\le 1/\sqrt{2}<\infty$, whereas $\Xi(\Phi,\Phi;1,0)<\infty$ if and only if $\nu$ is dyadically doubling.

\section{Further Results}\label{section:further}

\subsection{Non-cancellative Haar shift operators}\label{section:NC-HS}
One can consider Haar shift operators defined in terms of generalized Haar systems that are not required to satisfy the vanishing integral condition. To elaborate on this, let us first consider the case of the dyadic paraproducts and their adjoints.  The space $\BMO_\dd(\mu)$ is the space of locally integrable functions $\rho$ such that
$$
\|\rho\|_{\BMO_\dd(\mu)} = \sup_{Q \in \dd} \Big(\frac1{\mu(Q)}\int_Q \big|\rho(x) - \langle \rho \rangle_Q \big|^2 \,d\mu(x)\Big)^{\frac12} < \infty,
$$
where as usual the terms where $\mu(Q)=0$ are assumed to be 0.
Given $\rho\in\BMO_\dd(\mu)$, and $\Theta=\{\theta_Q\}_{Q\in\dd}$, $\Psi=\{\psi_Q\}_{Q\in\dd}$, two (cancellative) generalized Haar systems, we define the \emph{dyadic paraproduct} $\Pi_\rho$:
$$
\Pi_\rho f(x)
=
\sum_{Q \in \dd}\langle \rho,\theta_Q \rangle \langle f \rangle_Q \psi_{Q}(x).
$$
Note that for each cube $Q$,  $\theta_Q$ and $\psi_Q$ are cancellative generalized Haar functions. However, the term $\langle f \rangle_Q$ can be viewed, after renormalization,  as $f$ paired with the non-cancellative generalized Haar function $1_Q/\mu(Q)^{1/2}$. That is the reason why we call this operator a non-cancellative Haar shift, see below for further details.

Alternatively, one can consider dyadic paraproducts by incorporating $\mu$-Carleson sequences. Given a sequence $\gamma=\{\gamma_Q\}_{Q\in\dd}$, we say that $\gamma$ is a $\mu$-Carleson sequence, which is denoted by $\gamma\in\car$, if for every $Q\in\dd$ we have that $\gamma_Q=0$ if $\mu(Q)=0$ and
$$
\|\gamma\|_{\car}
=
\sup_{Q\in\dd,\, \mu(Q)>0} \left(\frac{\sum_{Q'\in\dd(Q)}|\gamma_{Q'}|^2}{\mu(Q)}\right)^{\frac12}<\infty.
$$
Typical examples of $\mu$-Carleson sequences are given by $\BMO(\mu)$ functions. Indeed if $\rho\in\BMO_\dd(\mu)$, $\Theta=\{\theta_Q\}_{Q\in\dd}$ is a generalized Haar system and we set $\gamma_Q=\langle\rho,\theta_Q \rangle$ we have that $\gamma$ is $\mu$-Carleson measure: if $Q_0\in\dd$ such that $\mu(Q_0)>0$, we have by orthogonality
\begin{multline*}
\sum_{Q\in\dd(Q_0)}|\gamma_Q|^2
=
\sum_{Q \in \dd(Q_0)} \big|\big\langle (\rho-\langle \rho\rangle_{Q_0})1_{Q_0},\theta_Q \big\rangle\big|^2
\\
\le
\|(\rho-\langle \rho\rangle_{Q_0})1_{Q_0}\|_{L^2(\mu)}^2
\le
\|\rho\|_{\BMO(\mu)}^2\,
\mu(Q_0)
\end{multline*}
and therefore $\|\gamma\|_{\car}\le \|\rho\|_{\BMO(\mu)}$. One can also reverse this procedure. Indeed, given $\gamma\in\car$ and a generalized Haar system $\Theta=\{\theta_Q\}_{Q\in\dd}$  we can define a function $\rho$ which is a Haar expansion using $\Theta$ with the coefficients  given by the sequence $\gamma$ as follows. It suffices to consider the function $\rho$ in any $d$-dimensional quadrant, say for simplicity that we are in $\R_1^d=[0,\infty)^d$. Let $Q_k=[0,2^{-k})^d$ and set
$$
\rho(x)=\sum_{k\in\Z}\Big(\sum_{Q\in\dd(Q_k)\setminus\dd(Q_{k+1})}\gamma_Q\,\theta_Q(x)\Big)1_{Q_k\setminus Q_{k+1}}(x).
$$
Note that for every $x\in\R^d_1$, the sum in $k$ contains only one non-vanishing term. From orthogonality and the Carleson condition it follows that for every $k_0\in\Z$,
\begin{multline}\label{eqn:norm-L2-rho}
\|\rho\|_{L^2(Q_{k_0})}^2
\le
\sum_{k\ge k_0}
\sum_{Q\in\dd(Q_k)\setminus\dd(Q_{k+1})}|\gamma_Q|^2
\\
=
\sum_{Q\in\dd(Q_{k_0})} |\gamma_Q|^2
\le
\|\gamma\|_{\car}^2\,
\mu(Q_{k_0}).
\end{multline}
In particular $\rho$ is locally integrable. We next take an arbitrary $R\in\dd$, $R\subset\R_1^d$. Assume first that $R=Q_{k_0}$ for some $k_0\in\Z$. Then easy calculations and \eqref{eqn:norm-L2-rho} lead to
\begin{multline*}
\frac1{\mu(R)}\int_R \big|\rho(x) - \langle \rho \rangle_R \big|^2 \,d\mu(x)
=
\langle |\rho|^2 \rangle_{Q_{k_0}}-\big|\langle \rho \rangle_{Q_{k_0}}|^2
\\
\le
\langle |\rho|^2 \rangle_{Q_{k_0}}
=
\mu(Q_{k_0})^{-1}\,\|\rho\|_{L^2(Q_{k_0})}^2
\le
\|\gamma\|_{\car}^2.
\end{multline*}
On the other hand if $R\notin \{Q_k\}_k$, then there exists a unique $k$ such that $R\subset Q_k\setminus Q_{k+1}$. Then for every $x\in R$ we have
\begin{multline*}
\rho(x)
=
\sum_{Q\in\dd(Q_k)\setminus\dd(Q_{k+1})}\gamma_Q\,\theta_Q(x)
\\
=
\sum_{Q\in\dd(R)}\gamma_Q\,\theta_Q(x)+\sum_{\begin{subarray}{c}
Q\in\dd(Q_k)\setminus\dd(Q_{k+1})
\\
R\subsetneq Q
\end{subarray}}\gamma_Q\,\theta_Q(x)
=I(x)+II.
\end{multline*}
Note that $II$ is constant and that $\int_R I(x)\,d\mu(x)=0$ then
\begin{multline*}
\frac1{\mu(R)} \int_{R} |\rho(x)-\langle\rho\rangle_R|^2\,d\mu(x)
=
\frac1{\mu(R)}\int_{R} \Big|\sum_{Q\in\dd(R)}\gamma_Q\,\theta_Q(x)\Big|^2\,d\mu(x)
\\
\le
\frac1{\mu(R)}\sum_{Q\in\dd(R)}|\gamma_Q|^2
\le
\|\gamma\|_{\car}^2.
\end{multline*}
Gathering the two cases it follows that $\rho\in \BMO(\mu)$ with $\|\rho\|_{\BMO(\mu)}\le \|\gamma\|_{\car}$. Further details are left to the reader.

Given $\gamma$ a $\mu$-Carleson sequence and $\Psi=\{\psi_Q\}_{Q\in\dd}$ a generalized Haar system we define the \emph{dyadic paraproduct} $\Pi_\gamma$ as follows
$$
\Pi_\gamma f(x)
=
\sum_{Q \in \dd}\gamma_Q\, \langle f \rangle_Q \psi_{Q}(x).
$$
If we set $\widetilde{\phi}_Q=1_{Q}/\mu(Q)^{1/2}$ if $\mu(Q)>0$ and $\widetilde{\phi}_Q\equiv 0$ otherwise we have that $\widetilde{\Phi}=\{\widetilde{\phi}_Q\}_{Q\in\dd}$ satisfies $(a)$, $(b)$ and $(d)$ in Definition \ref{def:GHS}. Since $(c)$ does not hold we call $\widetilde{\Phi}$ a non-cancellative generalized Haar system. In such a way we can write
$$
\Pi_\gamma f(x)
=
\sum_{Q \in \dd}\alpha_Q\,\langle f,\widetilde{\phi}_Q \rangle \psi_{Q}(x),
\qquad
\alpha_Q=\frac{\gamma_Q}{\mu(Q)^{\frac12}}.
$$
Note that
$$
|\alpha_Q|
\le
\left(\frac{\sum_{Q'\in\dd(Q)}|\gamma_{Q'}|^2}{\mu(Q)}\right)^{\frac12}
\le
\|\gamma\|_{\car}.
$$
Thus, we can see $\Pi_\gamma$ as a Haar shift of complexity $(0,0)$ with respect to the non-cancellative generalized Haar system $\widetilde{\Phi}$ and
the (cancellative) generalized Haar system $\Psi$. Notice that the adjoint of the paraproduct can be written as
$$
\Pi_\gamma^* f(x)
=
\sum_{Q \in \dd} \gamma_Q\,\langle f, \psi_{Q}\rangle\,\frac{1_Q(x)}{\mu(Q)}
=
\sum_{Q \in \dd}\alpha_Q\,\langle f,\psi_Q\rangle \widetilde{\phi}_Q (x).
$$
Again $\Pi_\gamma^*$ is a Haar shift of complexity $(0,0)$ with respect to a (cancellative) generalized Haar system $\Psi$ and the non-cancellative generalized Haar system $\widetilde{\Phi}$. This motivates the definition of a non-cancellative Haar shift operator:
\begin{equation}\label{non-cancel-HS}
\widetilde{\Sha}_{r,s} f(x)
=
\sum_{Q \in \dd } \sum_{\begin{subarray}{c} R \in \dd _r(Q) \\ S \in \dd _s(Q) \end{subarray}} \alpha_{R,S}^Q \langle f, \widetilde{\phi}_R \rangle \widetilde{\psi}_S(x), \qquad \sup_{Q,R,S}|\alpha_{R,S}^Q|<\infty,
\end{equation}
with $\widetilde{\Phi}=\{\widetilde{\phi}_Q\}_{Q\in\dd}$ and $\widetilde{\Psi}=\{\widetilde{\psi}_Q\}_{Q\in\dd}$  being two non-cancellative generalized Haar systems, i.e., both of them satisfies  $(a)$, $(b)$ and $(d)$ in Definition \ref{def:GHS}. We would like to stress that $\widetilde{\Phi}$ and $\widetilde{\Psi}$ do not necessarily satisfy $(c)$, therefore the $L^2(\mu)$ boundedness does not automatically follow from the assumed conditions. Thus, is natural to impose that $\widetilde{\Sha}_{r,s}$ is bounded on $L^2(\mu)$ along with some local boundedness property and these condition will be checked in any specific situation.

\begin{theorem}\label{thm:HS-general:non-cancel}
Let $\mu$ be a Borel measure on $\R^d$, $d\ge 1$, satisfying that $\mu(Q)<\infty$ for every $Q\in\dd $.
Let $\widetilde{\Phi}=\{\widetilde{\phi}_Q\}_{Q\in\dd}$ and $\widetilde{\Psi}=\{\widetilde{\psi}_Q\}_{Q\in\dd}$ be two non-cancellative generalized Haar systems in $\R^d$. Let $r$, $s$ be two non-negative integers
and consider $\widetilde{\Sha}_{r,s}$ as in \eqref{non-cancel-HS}. Assume that $\widetilde{\Sha}_{r,s}$ is bounded on $L^2(\mu)$ and also that $\widetilde{\Sha}_{r,s}$ satisfies the following restricted local $L^2(\mu)$ boundedness: for every $Q_0\in\dd$ we have that
\begin{equation}\label{non-cancel-HS-local-L2}
\|\widetilde{\Sha}_{r,s}^{Q_0}(1_{Q_0})\|_{L^2(\mu)}
\lesssim
\mu(Q_0)^{\frac12},
\end{equation}
where the constant is uniform on $Q_0$ and
$$
\widetilde{\Sha}_{r,s}^{Q_0} f(x)
=
\sum_{Q \in \dd(Q_0) } \sum_{\begin{subarray}{c} R \in \dd _r(Q) \\ S \in \dd _s(Q) \end{subarray}} \alpha_{R,S}^Q \langle f, \widetilde{\phi}_R \rangle \widetilde{\psi}_S(x).
$$
If $\Xi(\widetilde{\Phi},\widetilde{\Psi}; r,s)<\infty$, then $\widetilde{\Sha}_{r,s}$ maps continuously $L^1(\mu)$ into $L^{1,\infty}(\mu)$.
\end{theorem}

\begin{remark}\label{remark:truncation}
Let us observe that $\widetilde{\Sha}_{r,s}^{Q_0}$ is the non-cancellative Haar shift operator associated with the sequence $\gamma^Q_{R,S}=\alpha^Q_{R,S}$ for $Q\in\dd(Q_0)$, $R \in \dd _r(Q)$,  $S \in \dd _s(Q)$; and $\gamma^Q_{R,S}=0$ otherwise. Also,  the $L^2(\mu)$ boundedness of $\widetilde{\Sha}_{r,s}^{Q_0}$  clearly implies \eqref{non-cancel-HS-local-L2}.
\end{remark}

\begin{remark}
Notice that if we further assume that both Haar systems $\widetilde{\Phi}$ and $\widetilde{\Psi}$ are cancellative, then we automatically obtain \eqref{non-cancel-HS-local-L2} and the $L^2(\mu)$ boundedness of $\widetilde{\Sha}_{r,s}^{Q_0}$ (see Section \ref{section:HS-general}). In such a case Theorem \ref{thm:HS-general:non-cancel} becomes Theorem \ref{thm:HS-general}
\end{remark}

\begin{proof}
The proof is similar to that of Theorem \ref{thm:HS-general}, therefore we only give the parts of the argument that are different. Again we may assume that $\mu\in\mathcal{M}_\infty$, the general case follows as before. Follow the proof of Theorem \ref{thm:HS-general}. For $S_1$ we use our assumption that $\widetilde{\Sha}_{r,s}$ is bounded on $L^2(\mu)$. The estimate for $S_2$ is the same. Let us observe that the estimate for $S_3$ is entirely analogous since in \eqref{eqn:Sha-bj} we have not used the vanishing integral of $\phi_Q$. We are then left with estimating $S_4$, for which we first observe that
\begin{multline*}
S_4
\le
\mu(\Omega_\lambda)+\mu\{ x\in \R^d\setminus\Omega_\lambda : |\widetilde{\Sha}_{r,s}  \beta(x)| > \lambda/3\}
\\
\le
\frac1{\lambda}\,\|f\|_{L^1(\mu)}
+
\frac{3}{\lambda}\sum_j\Big( \int_{\R^d\setminus \widehat{Q}_j} |\widetilde{\Sha}_{r,s}  \beta_j| d\mu + \int_{\widehat{Q}_j\setminus Q_j} |\widetilde{\Sha}_{r,s}  \beta_j| d\mu\Big)
\end{multline*}
and we estimate each term in the interior sum. Proceeding as in \eqref{eqn:Sha-bj} and using Theorem \ref{theor:CZ-d}  we can analogously obtain
\begin{multline*}
\sum_j\int_{\R^d\setminus  \widehat{Q}_j} |\widetilde{\Sha}_{r,s}\beta_j|\,d\mu
\\
\lesssim
\sum_j
\sum_{\widehat{Q}_j\subsetneq Q\subset Q_j^{(r+1)}}\sum_{\begin{subarray}{c} R \in \dd _r(Q), R\subset \widehat{Q}_j \\ S \in \dd _s(Q) \end{subarray}} \|\beta_j\|_{L^1(\mu)}\,\|\widetilde{\phi}_R\|_{L^{\infty}(\mu)}\|\widetilde{\psi}_S\|_{L^1(\mu)}
\\
\le
2^{2+(r+s)\,d}\,r\, \Xi(\widetilde{\Phi},\widetilde{\Psi}; r,s)\|f\|_{L^1(\mu)}.
\end{multline*}
On the other hand, for every $x\in \widehat{Q}_j\setminus Q_j$ we have
$$
|\widetilde{\Sha}_{r,s}\beta_j(x)|
\le
\bigg|\sum_{Q\subsetneq\widehat{Q}_j}\sum_{\begin{subarray}{c} R \in \dd _r(Q)\\ S \in \dd _s(Q) \end{subarray}} \dots \bigg|
+
\bigg|\sum_{\widehat{Q}_j\subset Q}\sum_{\begin{subarray}{c} R \in \dd _r(Q)\\ S \in \dd _s(Q) \end{subarray}} \dots \bigg|
=
F_j(x)
+
G_j(x)
$$
and we estimate each function in turn. For $F_j(x)$ we note that the terms $Q\subset Q_j$ vanish and therefore $R\subset Q\subset \widehat{Q}_j\setminus Q_j$. Thus $\beta_j$ is constant on $R$ and then
\begin{align*}
F_j(x)
&=
\bigg|
\sum_{Q\subsetneq \widehat{Q}_j\setminus Q_j}\sum_{\begin{subarray}{c} R \in \dd _r(Q)\\ S \in \dd _s(Q) \end{subarray}} \alpha_{R,S}^Q\langle \beta_j, \widetilde{\phi}_R\rangle\, \widetilde{\psi}_S(x)\bigg|
\\
&=
\big|\langle f \rangle_{Q_j} - \langle f \rangle_{\widehat{Q}_j}\bigl| \,\frac{\mu(Q_j)}{\mu(\widehat{Q}_j)}\,
\bigg|
\sum_{Q\subsetneq \widehat{Q}_j\setminus Q_j}\sum_{\begin{subarray}{c} R \in \dd _r(Q)\\ S \in \dd _s(Q) \end{subarray}} \alpha_{R,S}^Q\langle 1_{\widehat{Q}_j\setminus Q_j}, \widetilde{\phi}_R\rangle\, \widetilde{\psi}_S(x)\bigg|
\\
&\le
2\,\langle |f| \rangle_{Q_j}\,
\frac{\mu(Q_j)}{\mu(\widehat{Q}_j)}\,
\sum_{\begin{subarray}{c}Q'\in\dd_1(\widehat{Q}_j) \\ Q'\neq Q_j \end{subarray}}
\bigg|
\sum_{Q\in\dd(Q')}\sum_{\begin{subarray}{c} R \in \dd _r(Q)\\ S \in \dd _s(Q) \end{subarray}} \alpha_{R,S}^Q\langle 1_{Q'}, \widetilde{\phi}_R\rangle\, \widetilde{\psi}_S(x)\bigg|
\\
&=
2\,\langle |f| \rangle_{Q_j}\,
\frac{\mu(Q_j)}{\mu(\widehat{Q}_j)}\,
\sum_{\begin{subarray}{c}Q'\in\dd_1(\widehat{Q}_j) \\ Q'\neq Q_j \end{subarray}} \big|\widetilde{\Sha}_{r,s}^{Q'}(1_{Q'})(x)\big|.
\end{align*}
This, the fact that $\supp \widetilde{\Sha}_{r,s}^{Q'}(1_{Q'})\subset Q'$ and  that these cubes and pairwise disjoint, and  \eqref{non-cancel-HS-local-L2} yield
\begin{multline*}
\int_{\widehat{Q}_j\setminus Q_j} F_j\,d\mu
\le
2\,\langle |f| \rangle_{Q_j}\,
\bigg(\frac{\mu(Q_j)}{\mu(\widehat{Q}_j)}\bigg)\,
\sum_{\begin{subarray}{c}Q'\in\dd_1(\widehat{Q}_j) \\ Q'\neq Q_j \end{subarray}} \int_{Q'} \big|\widetilde{\Sha}_{r,s}^{Q'}(1_{Q'})\big|d\mu
\\
\le
2\,\langle |f| \rangle_{Q_j}\,
\bigg(\frac{\mu(Q_j)}{\mu(\widehat{Q}_j)}\bigg)\,
\sum_{\begin{subarray}{c}Q'\in\dd_1(\widehat{Q}_j) \\ Q'\neq Q_j \end{subarray}} \big\|\widetilde{\Sha}_{r,s}^{Q'}(1_{Q'})\big\|_{L^2(\mu)}\,\mu(Q')^{\frac{1}{2}}
\lesssim
\int_{Q_j} |f|\,d\mu.
\end{multline*}
For $G_j$ we proceed as before
\begin{multline*}
\int_{\widehat{Q}_j\setminus Q_j} G_j\,d\mu
\lesssim
\sum_{\widehat{Q}_j\subset Q\subset Q_j^{(r+1)}}\sum_{\begin{subarray}{c} R \in \dd _r(Q), R\subset \widehat{Q}_j \\ S \in \dd _s(Q) \end{subarray}} \|\beta_j\|_{L^1(\mu)}\,\|\widetilde{\phi}_R\|_{L^{\infty}(\mu)}\|\widetilde{\psi}_S\|_{L^1(\mu)}
\\
\le
2^{(r+s)\,d}\,r\, \Xi(\widetilde{\Phi},\widetilde{\Psi}; r,s)\|\beta_j\|_{L^1(\mu)}.
\end{multline*}
Gathering the previous estimates we conclude that
\begin{multline*}
\sum_j\int_{\widehat{Q}_j\setminus Q_j} |\widetilde{\Sha}_{r,s}\beta_j|\,d\mu
\le
\sum_j\int_{\widehat{Q}_j\setminus Q_j} (F_j+ G_j)\,d\mu
\\
\lesssim
\sum_j \int_{Q_j} |f|\,d\mu + \sum_j \|\beta_j\|_{L^1(\mu)}
\lesssim
\|f\|_{L^1(\mu)}.
\end{multline*}
\end{proof}

\begin{remark}\label{remark:HS-NC-constant}
As above, if we keep track of the constants and use a standard homogeneity argument we obtain that, under the conditions of Theorem \ref{thm:HS-general},
\begin{multline*}
\|\widetilde{\Sha}_{r,s}\|_{L^1(\mu)\to L^{1,\infty}(\mu)}\!
\le
C_0\,\Big(\|\widetilde{\Sha}_{r,s}\|_{L^2(\mu)\to L^{2}(\mu)}
\\
+ \sup_{Q\in\dd,\mu(Q)\neq 0} \frac{\|\widetilde{\Sha}_{r,s}^{Q}(1_{Q})\big\|_{L^2(\mu)}}{\sqrt{\mu(Q)}}
+
2^{(s+r)\,d}\,r\,\,\Xi(\Phi,\Psi; r,s)\sup_{Q,R,S} |\alpha_{R,S}^Q|\Big),
\end{multline*}
where $C_0$ is a universal constant (independent of the dimension, for instance, in the previous argument one can safely take $C_0\le 220$).
\end{remark}

\subsection{Dyadic paraproducts}

As a consequence  of Theorem \ref{thm:HS-general:non-cancel} we can obtain the following result for dyadic paraproducts.

\begin{theorem}\label{theor:parap}
Let $\mu$ be a Borel measure on $\R^d$, $d\ge 1$, satisfying that $\mu(Q)<\infty$ for every $Q\in\dd $. Let  $\Psi=\{\psi_Q\}_{Q\in\dd}$ be a generalized Haar system. Given a sequence $\gamma=\{\gamma_Q\}_{Q\in\dd}$ we consider the \emph{dyadic paraproduct} $\Pi_\gamma$ and its adjoint $\Pi_\gamma^*$\textup{:}
$$
\Pi_\gamma f(x)
=
\sum_{Q \in \dd} \gamma_Q\,\langle f \rangle_Q \psi_{Q}(x),
\qquad
\Pi_\gamma^* f(x)
=
\sum_{Q \in \dd} \gamma_Q\,\langle f, \psi_{Q}\rangle\,\frac{1_Q(x)}{\mu(Q)}.
$$
Then we have the following\textup{:}
\begin{list}{$(\theenumi)$}{\usecounter{enumi}\leftmargin=1cm
\labelwidth=1cm\itemsep=0.2cm\topsep=.1cm
\renewcommand{\theenumi}{\roman{enumi}}}

\item For every $\gamma\in\car$, $\Pi_\gamma$ is of weak-type $(1,1)$ and there exists a universal constant $C_0$ \textup{(}one can take for instance $C_0\le 288$\textup{)} such that
$$
\|\Pi_\gamma f\|_{L^{1,\infty}(\mu)}
\le
C_0\,\|\gamma\|_{\car}\|f\|_{L^1(\mu)}.
$$
Consequently, $\Pi_\gamma$ is bounded on $L^p(\mu)$, $1<p\le 2$ \textup{(}the constant is dimension free and depends linearly on $\|\gamma\|_{\car}$\textup{)}.

\item If
\begin{equation}\label{comp-adjoint-parap}
\sup_{Q\in\dd} \|\psi_{Q}\|_{L^{\infty}(\mu)}\,\mu(Q)^{\frac12}<\infty,
\end{equation}
then $\Pi_\gamma^*$ is of weak-type $(1,1)$ for every $\gamma\in\car$ with boundedness constant depending linearly on $\|\gamma\|_{\car}$. Conversely,  if $\Pi_\gamma^*$ is of weak-type $(1,1)$ with $\|\Pi_\gamma^*\|_{L^1(\mu)\to L^{1,\infty}} \le C\,\|\gamma\|_{\car}$ for every $\gamma\in \car$, then \eqref{comp-adjoint-parap} holds. Additionally,
if \eqref{comp-adjoint-parap} holds then $\Pi_\gamma^*$ is bounded on $L^p(\mu)$ for $1<p<2$ \textup{(}the case $p\ge 2$ follows from $(i)$ without assuming \eqref{comp-adjoint-parap}\textup{)}.

\item Suppose in particular that $d=1$, $\mu(I)>0$ for every $I\in\dd$ and that $\Psi=\mathcal{H}$. Then, $\Pi_\gamma$ is of weak-type $(1,1)$ and bounded on $L^p(\mu)$, $1<p\le 2$, for every $\gamma\in \car$. However, if for every $\gamma\in \car$ we have that  $\Pi_\gamma^*$ is of weak-type $(1,1)$ or weak-type $(p,p)$ for some $1<p<2$, then $\mu$ is dyadically doubling. Conversely, if $\mu$ is dyadically doubling then $\Pi_\gamma^*$ is of weak-type $(1,1)$ and bounded on $L^p(\mu)$, $1<p<2$, for every $\gamma\in \car$.

\item In $(i)$, $(ii)$, $(iii)$ we can replace the condition ``$\gamma\in\car$'' by  ``$\gamma_Q=\langle\rho,\theta_Q\rangle$ with $\rho\in\BMO(\mu)$ and $\Theta=\{\theta_Q\}_{Q\in\dd}$ a generalized Haar system''; and in the boundedness constants $\|\gamma\|_{\car}$  by $\|\rho\|_{\BMO(\mu)}$.

\end{list}
\end{theorem}

Before starting the proof, let us state the $L^2(\mu)$ boundedness of the paraproduct (and its adjoint) along with the corresponding restricted local  boundedness as a lemma:

\begin{lemma}\label{lemma:L2-parpap}
Under the assumptions of Theorem \ref{theor:parap}, for every $\gamma\in \car$ we have
\begin{equation}\label{L2-parpap}
\|\Pi_\gamma f\|_{L^2(\mu)}
\le
2\,\|\gamma\|_{\car}\,\|f\|_{L^2(\mu)}.
\end{equation}
Moreover, for every $Q_0\in\dd$ we obtain
\begin{equation}\label{L2-parpap-local}
\|\Pi_\gamma^{Q_0} f\|_{L^2(\mu)}
\le
2\,\|\gamma\|_{\car}\,\|f\|_{L^2(\mu)},
\quad
\Pi_\gamma^{Q_0} f
=
\!\!\!\sum_{Q \in \dd(Q_0)}\!\!\! \gamma_Q\,\langle f \rangle_Q \psi_{Q}.
\end{equation}
\end{lemma}

\begin{proof}
We claim that it suffices to obtain \eqref{L2-parpap}. Indeed, we consider a new sequence $\widetilde{\gamma}=\{\widetilde{\gamma}_Q\}_{Q\in\dd}$ with $\widetilde{\gamma}_Q=\gamma_Q$ if $Q\in\dd(Q_0)$ and $\widetilde{\gamma}_Q=0$ otherwise. We clearly have that $\widetilde{\gamma}\in\car$ with $\|\widetilde{\gamma}\|_{\car}\le \|\gamma\|_{\car}$ and also $\Pi_\gamma^{Q_0}=\Pi_{\widetilde{\gamma}}$. Thus, \eqref{L2-parpap} applied to $\widetilde{\gamma}$ implies \eqref{L2-parpap-local}.

We obtain \eqref{L2-parpap} using ideas from \cite{pereyra2000}. Let us first suppose that $\mu\in\mathcal{M}_\infty$. The argument is somehow standard, but, since our setting is very general, we give the argument for completeness. Given $f\in L^2(\mu)$ and $\lambda>0$, as in Theorem \ref{theor:CZ-d}, we can find a maximal collection of dyadic cubes $\{Q_j^\lambda\}_j$ such that $\Omega_\lambda=\cup_{j} Q_j^\lambda$. We notice that the existence of such maximal cubes follows from the fact that $\langle|f|\rangle_Q\le \langle|f|^2\rangle_Q^{1/2}\to 0$ as $\ell(Q)\to\infty$, given our current assumption $\mu\in \mathcal{M}_\infty$. Next we use that $\Psi$ is cancellative, therefore orthogonal,
\begin{multline}\label{trick-parap}
\|\Pi_\gamma f\|_{L^2(\mu)}^2
=
\sum_{Q\in\dd}
|\gamma_Q|^2\,\big|\langle f \rangle_Q\big|^2 \|\psi_{Q}\|_{L^2(\mu)}^2
\\
\le
\int_0^\infty \sum_{Q \in \dd} 1_{\{\langle|f|\rangle_Q>\lambda \}}(\lambda)\,|\gamma_Q|^2\, 2\,\lambda\, d\lambda
\leq \int_0^\infty \sum_j\!\sum_{Q \in \dd(Q_j^{\lambda})}\!\! |\gamma_Q|^2\, 2\lambda\, d\lambda.
\\
\leq
\|\gamma\|_{\car}^2\int_0^\infty \sum_j\mu(Q_j^\lambda)\, 2\lambda\, d\lambda
\le
\|\gamma\|_{\car}^2
\int_0^\infty  \,\mu(\Omega_\lambda)\,2\,\lambda\, d\lambda
\\
=
\|\gamma\|_{\car}^2\,\|M_\dd f\|_{L^2(\mu)}^2
\le
4\,\|\gamma\|_{\car}^2\,\|f\|_{L^2(\mu)}^2,
\end{multline}
and this completes the proof of the fact that $\Pi_\gamma$ is bounded on $L^2(\mu)$ provided $\mu\in\mathcal{M}_\infty$. To consider the general case, as before we may suppose that $\supp f\subset \R^d_k$, $1\le k\le 2^d$ with $\mu(\R^d_k)<\infty$. In \eqref{trick-parap} we split the integral in two: $0<\lambda\le \langle|f|\rangle_{\R^d_k}$ and $\lambda> \langle|f|\rangle_{\R^d_k}$. In the second case we can find the maximal cubes $\{Q_j^\lambda\}$ and the previous argument goes through. Let us next consider the integral in the range  $0<\lambda\le \langle|f|\rangle_{\R^d_k}$. Let $\{Q_n\}_{n\ge 1}\subset\dd(\R^d_k)$ be an increasing sequence such that $\cup_n Q_n=\R^d_k$. Then,
we proceed as above
\begin{multline*}
\int_0^{\langle|f|\rangle_{\R^d_k}} \sum_{Q \in \dd(\R^d_k)} 1_{\{\langle|f|\rangle_Q>\lambda \}}(\lambda)\,|\gamma_Q|^2\, 2\,\lambda\, d\lambda
\le
\langle|f|\rangle_{\R^d_k}^2\,\sup_n\sum_{Q \in \dd(Q_n)}|\gamma_Q|^2
\\
\le
\|\gamma\|_{\car}^2\,
\langle|f|^2\rangle_{\R^d_k}\,
\sup_n \mu(Q_n)
=
\|\gamma\|_{\car}^2\,\|f\|_{L^2(\mu)}^2.
\end{multline*}
This completes the proof of \eqref{L2-parpap}.
\end{proof}

\begin{proof}[Proof of Theorem \ref{theor:parap}]
We start with $\Pi_{\gamma}$. Set $\widetilde{\phi}_Q=1_Q/\mu(Q)^{1/2}$ if $\mu(Q)>0$ and $\widetilde{\phi}_Q=0$ otherwise and consider the non-cancellative generalized Haar system $\widetilde{\Phi}=\{\widetilde{\phi}_Q\}_{Q\in\dd}$. As explained above, in the notation of Theorem \ref{thm:HS-general:non-cancel}, $\Pi_\gamma$ is a non-cancellative Haar shift operator of complexity $(0,0)$ with respect to the systems $\widetilde{\Phi}$ and $\widetilde{\Psi}=\Psi$. By Lemma \ref{lemma:L2-parpap} we have the required $L^2(\mu)$ bounds in Theorem \ref{thm:HS-general:non-cancel}. Thus the weak-type $(1,1)$ (and by interpolation the boundedness on $L^p(\mu)$, $1<p<2$) of $\Pi_\gamma$ follows from the property
$\Xi(\widetilde{\Phi},\Psi; 0,0)<\infty$. But this is in turn trivial: by H\"older's inequality we have for every $Q\in\dd_\Psi$
$$
\|\widetilde{\phi}_Q\|_{L^\infty(\mu)}\|\psi_Q\|_{L^1(\mu)}
=
\mu(Q)^{-\frac12}\|\psi_Q\|_{L^1(\mu)}
\le
\|\psi_Q\|_{L^2(\mu)}
=1.
$$
This completes the proof of $(i)$. For the boundedness constant we can use Remark \ref{remark:HS-NC-constant} along with Lemma \ref{lemma:L2-parpap} to obtain the linear dependence on  $\|\gamma\|_{\car}$.

We now turn to $(ii)$. We have shown that $\Pi_\gamma$ is bounded on $L^2(\mu)$ and so is its adjoint $\Pi_\gamma^*$. Notice that $(\Pi_\gamma^*)^{Q_0}=(\Pi_\gamma^{Q_0})^*$ and therefore $(\Pi_\gamma^*)^{Q_0}$ satisfies \eqref{L2-parpap-local}. Then, we apply again Theorem \ref{thm:HS-general:non-cancel} to $\Pi_\gamma^*$ which is a non-cancellative Haar shift operator of complexity $(0,0)$ with respect to the non-cancellative generalized Haar systems $\Psi$, $\widetilde{\Phi}$. Thus, $\Xi(\Psi, \widetilde{\Phi}; 0,0)<\infty$, which coincides with \eqref{comp-adjoint-parap}, implies that $\Pi_\gamma^*$ is of weak-type $(1,1)$ . The linear dependence on $\|\gamma\|_{\car}$  uses the same argument as above. Let us now obtain the converse. Notice that in \eqref{comp-adjoint-parap} we can restrict the supremum to $Q\in\dd_\Psi$ and in particular $\mu(Q)>0$. Fix one of these cubes $Q_0$ and let $\gamma_Q=\delta_{Q,Q_0}\,\sqrt{\mu(Q_0)}$. Then, $\gamma\in\car$ with $\|\gamma\|_{\car}=1$. Take
$$
f=\sgn\big(\psi_{Q_0}(x)\big)\frac{1_{Q_{0,\infty}}(x)}{\mu(Q_{0,\infty})},
$$
where $Q_{0,\infty}\in\dd_1(Q_0)$ is a cube where $\psi_{Q_0}$ attains its maximum. Then, as in the proof of Theorem \ref{thm:HS-general} and using that $\Pi_\gamma^*$ is of weak-type $(1,1)$ with uniform constant (since $\|\gamma\|_{\car}=1$) we obtain
\begin{multline*}
\|\psi_{Q_0}\|_{L^\infty(\mu)}\,\sqrt{\mu(Q_0)}
=
\Big\|\langle f, \psi_{Q_0}\rangle\,\frac{1_{Q_0}}{\sqrt{\mu(Q_0)}}\Big\|_{L^{1,\infty}(\mu)}
=
\|\Pi_\gamma^* f\|_{L^{1,\infty}(\mu)}
\\
\le
C\,\|f\|_{L^1(\mu)}
=
C.
\end{multline*}
Repeating this for every $Q_0\in\dd_\Psi$ we obtain \eqref{comp-adjoint-parap} as desired.

To complete the proof of $(ii)$ we first observe that for $p\ge 2$, duality and $(i)$ give the
$L^p(\mu)$ boundedness of $\Pi_\gamma^*$ with no further assumption on $\mu$. For $1<p<2$, assuming \eqref{comp-adjoint-parap}, we already know that $\Pi_\gamma^*$ is of weak-type $(1,1)$. The desired estimates now follow by interpolation with the $L^2(\mu)$ bound from Lemma \ref{lemma:L2-parpap}.

To obtain $(iii)$ we apply $(i)$ and $(ii)$ and observe that \eqref{comp-adjoint-parap} can be written as
$$
\sup_{I\in\dd} \frac{\sqrt{\mu(I)}}{\sqrt{\min\{\mu(I_-),\mu(I_+)\}}}
\approx
\sup_{I\in\dd} \frac{\sqrt{\mu(I)}}{\sqrt{m(I)}}<\infty,
$$
which in turn is equivalent to the fact that $\mu$ is dyadically doubling. To complete the proof of $(iii)$ it remains to show that if $\Pi_\gamma^*$ is of weak type $(p,p)$ for some $1<p<2$ then $\mu$ is dyadically doubling. Fix then $1<p<2$ and $I_0\in\dd$. Let $\gamma_I=\delta_{I,\widehat{I}_0}\,\sqrt{\mu(\widehat{I}_0)}$ and observe that $\gamma\in\car$ with $\|\gamma\|_{\car}=1$. Taking $f=h_{\widehat{I}_0}$, by \eqref{hI-Lp} we have
\begin{multline*}
\mu(\widehat{I}_0)^{\frac1p-\frac12}
=
\Big\|\gamma_{\widehat{I}_0}\,\frac{1_{\widehat{I}_0}}{\mu(\widehat{I}_0)}\Big\|_{L^{p,\infty}(\mu)}
=
\|\Pi_\gamma^* f\|_{L^{p,\infty}(\mu)}
\\
\le
C\,\|f\|_{L^p(\mu)}
\approx m(\widehat{I}_0)^{\frac12-\frac1{p'}}
=
m(\widehat{I}_0)^{\frac1p-\frac1{2}}
\le
\mu(I_0)^{\frac1p-\frac1{2}}.
\end{multline*}
This estimate holds for every $I_0\in\dd$ and therefore $\mu$ is dyadically doubling as desired.

We finally show $(iv)$. As observed before if we set $\gamma_Q=\langle\rho,\theta_Q\rangle$ with $\rho\in\BMO(\mu)$ and $\Theta=\{\theta_Q\}_{Q\in\dd}$ being a generalized Haar system we have that $\gamma\in\car$ with $\|\gamma\|_{\car}\le \|\rho\|_{\BMO(\mu)}$. Therefore the only assertion that is not contained in the previous items is the converse implication in $(ii)$. As before, in \eqref{comp-adjoint-parap}, we can restrict the supremum to $Q\in\dd_\Psi$ and in particular $\mu(Q)>0$. Fix one of these cubes $Q_0$, take $\Theta=\Psi$ and let $\rho=\psi_{Q_0}\,\sqrt{\mu(Q_0)}$. Then,
$$
\|\rho\|_{\BMO_\dd(\mu)}^2 = \sup_{Q_0\subset Q \in \dd} \frac1{\mu(Q)}\int_Q \big|\rho - \langle \rho \rangle_Q \big|^2 \,d\mu
=
\sup_{Q_0\subset Q \in \dd} \frac{\mu(Q_0)}{\mu(Q)}
=1.
$$
We take the same function $f$ as in $(ii)$, use that $\Pi_\rho^*$ is of weak-type $(1,1)$ with uniform constant (since $\|\rho\|_{\BMO_\dd(\mu)}=1$) and obtain
\begin{multline*}
\|\psi_{Q_0}\|_{L^{\infty}(\mu)}\,\mu(Q_0)^{\frac12}
=
\Big\|\langle f, \psi_{Q_0}\rangle\,\frac{1_{Q_0}}{\sqrt{\mu(Q_0)}}\Big\|_{L^{1,\infty}(\mu)}
=
\|\Pi_\rho^* f\|_{L^{1,\infty}(\mu)}
\\
\le
C\,\|f\|_{L^1(\mu)}
=C.
\end{multline*}
Repeating this for every $Q_0\in\dd_\Psi$ we obtain as desired \eqref{comp-adjoint-parap}. This completes the proof of $(iv)$.
\end{proof}

\subsection{On the probabilistic approach} \label{Paragraph-Probability}

We shall work with a fixed Borel measure $\mu$ on $\mathbb{R}^d$ such that $\mu(Q) < \infty$ for every dyadic cube $Q$.
The dyadic system $\mathscr{D} = (\mathscr{D}_k)_{k \in \mathbb{Z}}$ is a filtration on $\mathbb{R}^d$. The conditional expectation operator $\mathrm{E}_k$ associated to $\mathscr{D}_k$ is defined by
\begin{equation*}
\mathrm{E}_{k} f(x) = \sum_{Q \in \mathscr{D}_k}\mathrm{E}_Q f(x) = \sum_{Q \in \mathscr{D}_k} \langle f \rangle_Q 1_Q(x),
\end{equation*}
where $\langle f \rangle_Q = 0$ if $\mu(Q) = 0$. The martingale difference operators $\mathrm{D}_k$ are given by $\mathrm{D}_k = \mathrm{E}_k - \mathrm{E}_{k-1}$. It is clear from the definitions that the operators $\mathrm{E}_k$ form an increasing family projections that preserve integrals and that $\mathrm{D}_k$ are orthogonal projections. Thus, if $f \in L^p(\mu)$, $1\le p<\infty$, the sequence $(\mathrm{E}_k f)_{k \in \mathbb{Z}}$ is an $L^p$-martingale and
\begin{equation}
\label{eq:ortdiffn}
f(x) = \sum_{k \in \mathbb{Z}} \mathrm{D}_k f + \mathrm{E}_{-\infty}f = \sum_{n > k} \mathrm{D}_n f + \mathrm{E}_k f,
\end{equation}
where the convergence is in $L^p(\mu)$ and $\mu$-almost everywhere, and where $\mathrm{E}_{-\infty}f= \sum_{j=1}^{2^d} \langle f \rangle_{\mathbb{R}^d_j}1_{\mathbb{R}^d_j}$.
Let $Q \in \mathscr{D}_{k-1}$ and denote by $\mathrm{D}_Q$ the projection
$$
D_Q f(x) = \mathrm{D}_kf(x)\,1_Q(x) = \Big(\sum_{Q' \in \mathscr{D}_1(Q)} \mathrm{E}_{Q'}f(x)\Big) - \mathrm{E}_Qf(x).
$$
Hence $\mathrm{D}_k = \sum_{Q \in \mathscr{D}_{k-1}}\mathrm{D}_Q$. Observe that we may set $\mathrm{D}_Q f \equiv 0$ if $\mu(Q) = 0$. We easily obtain that  $\phi \in \mathrm{D}_Q(L^2(\mu))$ (by this we mean the image of $L^2(\mu)$ by the operator $\mathrm{D}_Q$) if and only if $\phi$ is supported on $Q$, constant on dyadic subcubes of $Q$, and has vanishing $\mu$-integral. In such a case we may write
\begin{equation}
\label{eq:DQL2}
\phi(x) = \sum_{Q' \in \mathscr{D}_1(Q)} a_{Q'}\frac{1_{Q'}(x)}{\mu(Q')},
\end{equation}
with $\sum_{Q' \in \mathscr{D}_1(Q)}a_{Q'} = 0$, and where it is understood that $a_Q'=0$ if $\mu(Q')=0$ and we use the standard convention that $0\cdot\infty=0$. Hence, $\mathrm{D}_Q(L^2(\mu))$ is a vector space of dimension at most $2^d-1$.

If we are in dimension $d=1$ and $I\in\mathscr{D}$ satisfies $\mu(I)>0$, then $h_I\in \mathrm{D}_I(L^2(\mu))$ (since $\mathrm{D}_Ih_I=h_I$). Note that in such a case $\mathrm{D}_I(L^2(\mu))$ is $1$-dimensional and therefore $\mathrm{D}_I f = \langle f,h_I \rangle h_I$, for every $f\in L^2(\mu)$.

In the higher dimensional case, assume for simplicity that $\mu(Q)>0$ for every $Q\in\mathscr{D}$. Let us consider the Wilson's Haar system $\{h_Q^\omega: \omega\in\mathscr{W}(Q),Q\in\mathscr{D}\}$. By othonormality of the Wilson's Haar system and the fact that the cardinality of $\mathscr{W}(Q)$ is $2^d-1$ we immediately  obtain that  $\{h_Q^\omega: \omega\in\mathscr{W}(Q)\}$ is an orthonormal basis  of $\mathrm{D}_Q(L^2(\mu))$. Thus,
$$
\mathrm{D}_Qf = \sum_{\omega \in \mathscr{W}(Q)} \langle f, h_Q^\omega \rangle h_Q^\omega,
\qquad
f\in L^2(\mu).
$$
The same can be done with Mitrea's Haar system (see above), in which case we obtain
$$
\mathrm{D}_Qf = \sum_{j=1}^{2^d-1} \langle f, H_Q^j \rangle H_Q^j,
\qquad
f\in L^2(\mu).
$$
Finally, if $\mu=\mu_1\times\dots\times\mu_d$ with $\mu_j$ Borel measures in $\R$ such that $0<\mu_j(I)<\infty$ for every $I\in\mathscr{D}(\R)$ and we consider the Haar system in the spirit of the wavelet construction $\{\phi_Q^\epsilon: \epsilon\in\{0,1\}^d\setminus\{0\}^d, Q\in\mathscr{D}\}$ we analogously have
$$
\mathrm{D}_Qf = \sum_{\epsilon\in\{0,1\}^d\setminus\{0\}^d} \langle f, \phi_Q^\epsilon \rangle \phi_Q^\epsilon,
\qquad
f\in L^2(\mu).
$$

We next see that martingale transforms can be written as Haar multipliers (i.e., Haar shifts of complexity $(0,0)$). A martingale transform is defined as
$$
Tf(x)=\sum_{k\in\Z} \xi_k(x)\,\mathrm{D}_k f(x)
$$
where the sequence $\{\xi_k\}_{k\in\Z}$ is predictable with respect to the dyadic filtration $(\mathscr{D}_k)_{k \in \mathbb{Z}}$, that is, $\xi_k$ is  $\sigma(\mathscr{D}_{k-1})$-measurable. Then $\xi_k$ is constant on the cubes $Q \in \mathscr{D}_{k-1}$. Namely, $\xi_k(x) = \sum_{Q \in \mathscr{D}_{k-1}}\alpha_Q 1_Q(x)$. Thus, by definition of the projections $\mathrm{D}_Q$ we get then that the martingale transform defined by $\{\xi_k\}_{k\in\Z}$ can be equivalently written as
\begin{equation}\label{eqn:martin-transform-hI}
Tf(x) = \sum_{k \in \mathbb{Z}}\sum_{Q \in \mathscr{D}_{k-1}}\alpha_Q\,\mathrm{D}_Q f(x)
=
\sum_{Q \in \mathscr{D}} \alpha_Q \Biggl(\,\sum_{j=1}^{2^{d}-1} \langle f,\psi_Q^{j} \rangle \psi_Q^{j}(x)\Biggr),
\end{equation}
with $\{\psi_Q^{j}\}_{1\le j\le 2^{d}-1}$ being any orthonormal basis of $\mathrm{D}_Q(L^2(\mu))$. Thus, every martingale transform can be represented as a sum of $2^{d}-1$ Haar multipliers, i.e., a Haar shift operators of complexity $(0,0)$ (see Example \ref{ex:Haar-mult}). Note that each Haar shift operator in the sum is written in terms of the system $\{\psi_Q^{j_Q}\}_{Q\in\mathscr{D}}$ where for each $Q\in\mathscr{D}$ we chose $j_Q$ with $1\le j_Q\le 2^d-1$.

It is easy to see that any orthonormal basis $\{\psi_Q^j\}_{1 \leq j \leq 2^d -1}$ of $\mathrm{D}_Q(L^2(\mu))$ is also a basis of $\mathrm{D}_Q(L^p(\mu))$ for $1 \leq p \leq \infty$. Assuming further that $\mu \in \mathcal{M}_\infty$, \eqref{eq:ortdiffn} says that $\{\psi_Q^j\}_{1 \leq j \leq 2^d -1, Q \in \mathscr{D}}$ is a basis of $L^p(\mu)$, $1 \leq p < \infty$. However, in view of \eqref{eqn:martin-transform-hI}, Burkholder's theorem of $L^p$ boundedness of martingale transforms, $1 <p < \infty$, does not suffice to show that a given Haar basis is unconditional in $L^p(\mu)$. In fact, unconditionality of a Haar basis is not true in general. We take the last example in Section \ref{section-Haar-systems} of a non-standard generalized Haar system and the Haar multiplier in \eqref{HM-non-stan}. We can easily see that for every $1<p<2$,
\begin{multline*}
\frac{\|T_\epsilon \varphi_{Q_k}\|_{L^{p}(\mu)}}{\|\varphi_{Q_k}\|_{L^p(\mu)}}
=
\frac{k\,\|\phi_{Q_k}\|_{L^{p}(\mu)}}{2\,\mu(Q_k^1)^{\frac1p-1}}
=
\frac{k}{2\,k^{2\,\frac{p-1}p}}\,\bigg(\frac{k^{p-2}}{2^{p-1}}+\bigg(\frac{k^2}{k^2-2}\bigg)^{\frac{p}2-1}\,2^{-\frac{p}2}\bigg)^\frac1p
\\
\gtrsim k^{\frac{2-p}{p}}
\underset{k \to \infty}{\longrightarrow} \infty.
\end{multline*}
Also, if we now take $\widetilde{\varphi}_{Q_k}=1_{Q_k^3}/\mu(1_{Q_k^3})$ then, for $2<p<\infty$,
\begin{multline*}
\frac{\|T_\epsilon \widetilde{\varphi}_{Q_k}\|_{L^{p}(\mu)}}{\|\widetilde{\varphi}_{Q_k}\|_{L^p(\mu)}}
=
\bigg(\frac{k^2}{2\,(k^2-2)}\bigg)^{\frac12}
\frac{\|\phi_{Q_k}\|_{L^{p}(\mu)}}{\mu(Q_k^3)^{\frac1p-1}}
\\
=
2^{\frac1p-\frac32}\bigg(\frac{k^2-2}{k^2}\bigg)^{\frac12-\frac1p}\,
\bigg(\frac{k^{p-2}}{2^{p-1}}+\bigg(\frac{k^2}{k^2-2}\bigg)^{\frac{p}2-1}\,2^{-\frac{p}2}\bigg)^\frac1p
\gtrsim k^{\frac{p-2}{p}}
\underset{k \to \infty}{\longrightarrow} \infty.
\end{multline*}
These imply that $\Phi=\{\phi_{Q_k}\}_{k\ge 2}$ is not an unconditional basis (on its span) on $L^p(\mu)$ for $1<p<\infty$ with $p\neq 2$.

Nevertheless, the standardness property
\begin{equation*}
\sup_{1 \leq j \leq 2^d -1}\sup_{Q \in \mathscr{D}} \|\psi_Q^j\|_{L^1(\mu)}\|\psi_Q^j\|_{L^\infty(\mu)}<\infty,
\end{equation*}
implies, by Theorem \ref{thm:HS-general}, that every Haar multiplier is of weak type $(1,1)$ and, by interpolation and duality, $L^p(\mu)$ bounded for every $1<p<\infty$.  This, in turn, gives that $\{\psi_Q^j\}_{1 \leq j \leq 2^d -1, Q \in \mathscr{D}}$ is an unconditional  basis for $L^p(\mu)$, $1<p<\infty$.

Let us now look at the case of the dyadic Hilbert transform an its adjoint in dimension $d=1$. Assume that $\mu(I)>0$ for every $I\in\mathscr{D}$. One can easily see that
$$
h_{I_\pm}(x) =  \mp\frac{\mu(I_\pm)}{\sqrt{m(I)}} h_{I_\pm}(x) h_I(x).
$$
Hence,
\begin{align*}
&H_\mathscr{D} f(x) = \sum_{I \in \mathscr{D}} \langle f,h_I \rangle \bigl(h_{I_-}(x) - h_{I_+}(x)\bigr) \\
&\qquad= \sum_{I \in \mathscr{D}} \frac{1}{\sqrt{m(I)}} \bigl(\mu(I_-) h_{I_-}(x) + \mu(I_+) h_{I_+}(x)\bigr) \langle f,h_I \rangle h_I(x) \\
&\qquad= \sum_{k \in \mathbb{Z}} \Biggl(\,\sum_{I \in \mathscr{D}_{k-1}} \frac{\mu(I_-) h_{I_-}(x) + \mu(I_+) h_{I_+}(x)}{\sqrt{m(I)}}\Biggr)\Biggl(\,\sum_{J \in \mathscr{D}_{k-1}}\langle f,h_J \rangle h_J(x)\Biggr) \\
&\qquad= \sum_{k \in \mathbb{Z}} \xi_{k}(x) \mathrm{D}_k f(x),
\end{align*}
where we have used that $\mathrm{D}_k = \sum_{I \in \mathscr{D}_{k-1}}\mathrm{D}_I$ and that $\mathrm{D}_I f = \langle f,h_I \rangle h_I$. The coefficient $\xi_{k}$ is $k+1$-measurable, defining a non predictable sequence. One may thus regard the dyadic Hilbert transform as a ``generalized martingale transform''.  Let us finally observe that for the adjoint of the Hilbert transform, since $\mathrm{D}_k$ is a projection, we have
\begin{equation*}
H_\mathscr{D}^*f(x) = \sum_{k \in \mathbb{Z}}  \mathrm{D}_k \bigl(\xi_{k}f\bigr)(x).
\end{equation*}
Similar expressions can be obtained for other Haar shift operators in every dimension provided the coefficients can be split as $\alpha_{R,S}^Q=\gamma_R^Q\,\beta_S^Q$. This procedure shows that Haar shift operators of arbitrary complexity ``fill'' the space of ``martingale transforms'' with arbitrary measurable coefficients,
further details are left to the interested reader. In particular, we see why classical tools coming from martingale $L^p$-theory do not apply in the present contexts, and our Calder\'on-Zygmund decomposition establishes the right substitute of Gundy's martingale decomposition in such a general setting.


\let\oldthebibliography=\thebibliography
  \let\endoldthebibliography=\endthebibliography
  \renewenvironment{thebibliography}[1]{\begin{oldthebibliography}{#1}\setlength{\parskip}{.1cm}\setlength{\itemsep}{0ex}}{\end{oldthebibliography}}

\end{document}